\documentclass[bibliography=totocnumbered]{scrartcl}
\usepackage[utf8]{inputenc}
\usepackage{hyperref}
\usepackage[english]{babel}
\usepackage{mathtools}
\usepackage{tensor}
\usepackage{enumerate}
\usepackage{aliascnt}
\usepackage{amssymb}
\usepackage{amsmath}
\usepackage{verbatim}

\usepackage{amsthm}

\allowdisplaybreaks[1] 

\newtheorem{prop}{Proposition}[section]

\newaliascnt{lemma}{prop} 
\newtheorem{lemma}[lemma]{Lemma}

\aliascntresetthe{lemma} 

\newaliascnt{mydef}{prop} 
 
\aliascntresetthe{mydef} 

\newaliascnt{corol}{prop} 
 \newtheorem{corol}[corol]{Corollary}
 
\aliascntresetthe{corol}

\newaliascnt{remark}{prop} 
 \newtheorem{remark}[remark]{Remark}
 
\aliascntresetthe{remark} 

\newaliascnt{mythm}{prop} 
 \newtheorem{mythm}[mythm]{Theorem}
 
\aliascntresetthe{mythm} 

\def\equationautorefname~#1\null{%
  (#1)\null
}

\newcommand{\N}{\ensuremath{\mathbb{N}}}
\newcommand{\Z}{\ensuremath{\mathbb{Z}}}
\newcommand{\Sph}{\ensuremath{\mathbb{S}}}
\newcommand{\Hyp}{\ensuremath{\mathbb{H}}}
\newcommand{\R}{\ensuremath{\mathbb{R}}}
\newcommand{\E}{\ensuremath{\mathcal{E}}}

\newcommand*\diff{\mathop{}\!\mathrm{d}}
\newcommand{\Kapp}{\ensuremath{\vec{\kappa}}}

\newcommand{\defeq}{\vcentcolon=}
\newcommand{\eqdef}{=\vcentcolon}
\newcommand{\CalC}{\ensuremath{\mathcal{C}}}
\newcommand{\Xcalnull}{\ensuremath{
{\mathcal{X}}{} }}
\newcommand{\Ycalnull}{\ensuremath{
{\mathcal{Y}}{} }}

\newcommand{\arccosh}{{\operatorname{arccosh}}}
\newcommand{\esssup}{{\operatorname{esssup}}}
\newcommand{\HS}[2][\ensuremath{4+\alpha}]
{\ensuremath{ H^{#1;#2} }}%

\newcommand{\st}{ \ensuremath{|}}
\title{The elastic flow of curves\\ 
in the hyperbolic plane}
\date{June 21, 2017}
\author{Anna Dall'Acqua\thanks{Universit\"at Ulm, Helmholtzstra\ss e 18, 89081 Ulm, Germany. 
\texttt{anna.dallacqua@uni-ulm.de}}, Adrian Spener\thanks{Universit\"at Ulm, Helmholtzstra\ss e 18, 89081 Ulm, Germany. 
\texttt{adrian.spener@uni-ulm.de}}}
\begin{document}

\maketitle

\begin{abstract}
We consider closed curves in the hyperbolic space
 moving by the $L^2$-gradient flow of the elastic energy and prove well-posedness and 
 long time existence. 
 Under the additional penalisation of the length we show subconvergence to critical points. A motivation for the study of this flow is given by the relation between elastic curves in the hyperbolic plane and Willmore  surfaces of revolution.
\end{abstract}


 \bigskip
 



\section{Introduction}

Let $f \colon \mathbb{S}^1 \to M$ be a smooth immersion of a closed curve 
in a smooth Riemannian $n$-dimensional manifold $(M^n,g)$ of constant sectional curvature $S_0$. In analogy to the Bernoulli model of an elastic rod in the Euclidean case we define its elastic energy as 
\begin{equation}
 \label{eq:defElasticEnergyRod}
 \E(f) = \frac{1}{2} \int_{\mathbb{S}^1} |\Kapp|^2 _g \diff s.
\end{equation}
Here $\diff s = |\partial_x f|_g \diff x$ 
 and the geodesic curvature $\Kapp$ is given as $ \Kapp = \nabla_{\partial_s f} \partial_s f$, 
i.e. the covariant derivative of $\partial_s f$ with respect to itself where $\partial_s f = \frac{1}{|\partial_xf|_g}\partial_xf \in \mathcal{T}(f)$ is the unit velocity vector field along $f$.

Critical points of the elastic energy subject to a length constraint or with a length penalisation are called  \textit{free elastica} and have been studied for instance in \cite{LangerSinger1}. These curves satisfy the equation
\begin{equation}
  \label{eq:gradientE}
  \nabla_{L^2}\E(f) = (\nabla_{\partial_s}^\bot)^2 \Kapp + \frac{1}{2} |\Kapp|^2_g \Kapp + S_0 \Kapp = 0 \,,
 \end{equation}
where $\nabla_{\partial_s}^\bot$ denotes the projection of the covariant derivative $\nabla_{\partial_s f}$ onto the subspace orthogonal to $\partial_s f$ (see \cite{LangerSinger1} and \autoref{rmk:proofGrad} below). \textit{Elastica} are critical point of the energy
\begin{equation}
 \label{eq:defElasticEnergy}
 \E_\lambda(f) = \frac{1}{2} \int_{\mathbb{S}^1} (|\Kapp|^2 _g + 2\lambda)\diff s ,
\end{equation}
with $\lambda>0$ and satisfy the equation
\begin{equation}
  \label{eq:gradientElambda}
  \nabla_{L^2}\E_{\lambda}(f) = (\nabla_{\partial_s}^\bot)^2 \Kapp + \frac{1}{2} |\Kapp|^2_g \Kapp - \lambda \Kapp + S_0 \Kapp = 0 \,.
 \end{equation}
Elastic curves are not only fundamental in the theory of mechanics and the calculus of variations (\cite{MR714991}), but have also modern applications, see for instance \cite{Mumford1994}.

In this work we study the gradient flow associated to the energy $\E_{\lambda}$ given by
\begin{equation}
  \label{eq:flow}
 \partial_t f =  -(\nabla_{\partial_s}^\bot)^2 \Kapp - \frac{1}{2} |\Kapp|^2_g \Kapp -S_0\Kapp + \lambda \Kapp ,
 \end{equation}
for sufficiently smooth immersions $f\colon \mathbb{S}^1 \times [0,T) \rightarrow M$.  A solution describes how an initial regular curve evolves in time reducing the energy $\E_\lambda$ in the direction of the steepest descent.
 This evolution has been studied in Euclidean space for instance in \cite{Wen,Koiso,DKS,OkabeNovaga}. Here we consider the case that $M$ is {two-dimensional} and the hyperbolic half-plane. In this situation, the sectional curvature is constant and equal to $-1$. It is quite natural to see if the negative curvature give rise to new phenomena. Another motivation for the study of the elastic flow in the hyperbolic space is given by the connection between elastic curves and Willmore surfaces of revolution, see Paragraph \ref{subs:revol} below. This connection has been used intensively to study various problems associated to the Willmore energy of surfaces of rotation with the elastic energy of curves in the hyperbolic space, for instance classification of rotational symmetric minimisers (\cite{LangerSinger1,LangerSinger2,BrGr}), blow-ups of Willmore surfaces of rotation (\cite{Blatt}) or the Dirichlet problem for rotational Willmore surfaces (\cite{MR2480063,MR2770424, MR2729304, 2017arXiv170502177M}).
 \\
Here we give a self-contained and complete description of the elastic flow of closed curves in the hyperbolic half plane. Below we state our main result.
 The similar case when the ambient manifold is $\mathbb{R}^n$ with the Euclidean metric has been studied in \cite{DKS} for closed curves and, for instance, in \cite{MR2911840,AnnaPaola1,ChunAnnaPaola}  for open curves.
\begin{mythm}\label{mainThm}
 Let $\Hyp^2$ be the hyperbolic half space, 
$f_0: \mathbb{S}^1\rightarrow \Hyp^2$ be a given smooth, regular and closed
 curve and $\lambda \geq 0$.
\begin{enumerate}[(i)]
 \item  There exists a smooth global solution $f:  \mathbb{S}^1 \times [0,\infty) \rightarrow \Hyp^2$ of the initial boundary value problem  \begin{equation}
  \label{eq:stat}
 \left\{ \begin{array}{ll}\partial_t f = - \nabla_{L^2}\E_\lambda(f) = -(\nabla_{\partial_s}^\bot)^2 \Kapp - \frac{1}{2} |\Kapp|^2_g \Kapp -S_0\Kapp + \lambda \Kapp , & \mbox{ in }\mathbb{S}^1 \times (0,T), \\
f(x, 0) = f_0(x), & \mbox{ for } x \in \mathbb{S}^1.
\end{array}\right. 
 \end{equation}
\item Moreover, if $\lambda>0$, as $t_{i} \to \infty$ there exists real values $p_i \in \R$, $\alpha_i > 0$ such that the curves $\alpha_i (f(t_{i, \cdot})-(p_i,0))$  subconverge, when reparametrised with constant speed, to a critical point of $\E_\lambda$, that is to a solution of \eqref{eq:gradientElambda}.
\item If $f_0$ is merely in $\CalC^{5,\alpha}$ there still exists a smooth solution on $(0,\infty)$ satisfying $f(\cdot, t) \to f_0$ 
in $\CalC^{1,\alpha}$ as $t \searrow 0$, and (ii) holds.
\end{enumerate}
\end{mythm}

The article is organised as follows: After introducing the necessary tools from hyperbolic geometry at the beginning of Section \ref{sec:2} 
we recapitulate the connection with the Willmore energy of surfaces of revolution in Paragraph \ref{subs:revol}. We finish this section with a description of the evolution of geometric quantities under the elastic flow. We devote Section \ref{section:STE} to the well-posedesness of \eqref{eq:stat} and show the long time existence and subconvergence in Section \ref{sec:LTE}. To improve the readability
of the paper but remain self-contained we have decided to collect some technical calculations in the Appendix.
\section{The hyperbolic plane and evolution of geometric quantities}
\label{sec:2}
In this section we compute the evolution equations of several geometric quantities in the hyperbolic plane. Here we choose to work in a general framework at first.

Let $(M^n,g)$ be a  (smooth) $n$-dimensional Riemannian manifold with local coordinates $(y_1, \ldots, y_n)$. By $\mathcal{T}(M)$ we denote the space of vector fields on $M$. As usual, for vector fields $X,Y \in  \mathcal{T}(M)$ we denote by $\nabla_X Y \in  \mathcal{T}(M)$
the unique connection on $M$ that is compatible with the metric and torsion free: the Levi-Civita connection. 
It can be expressed locally with the Christoffel symbols $\Gamma^k_{ij}$ as 
$  \nabla _{\partial_{y_i}} \partial_{y_j} = \sum_{k} \Gamma^k_{ij} \partial_{y_k}$.

If $\gamma \colon I \to M$ is a regular curve on $M$ and $X,Y$ vector fields along $\gamma$ then
\begin{equation}\label{eq:compcurves}
 \frac{\diff}{\diff t} \langle X(t),Y(t)\rangle_{g(\gamma(t))} = 
 \langle \nabla _{\dot \gamma} X(t),Y(t)\rangle_{g(\gamma(t))} +  \langle  X(t),\nabla _{\dot \gamma}Y(t)\rangle_{g(\gamma(t))} ,
\end{equation}
where locally,
 \begin{equation}\label{eq:CovDerLocally}
    \nabla _{\dot \gamma}X = \sum_{k=1}^n \left(\frac{\diff}{\diff t} X^{(k)}(t) + \sum_{i,j} X^{(j)}(t)\dot\gamma^{(i)}(t)\Gamma_{ij}^k(\gamma(t))\right)\partial_{y_k}.
 \end{equation}

Here we choose the following sign convention for the
Riemannian curvature tensor 
\begin{align*}
R & \colon  \mathcal{T}(M)\times \mathcal{T}(M)\times \mathcal{T}(M) \to \mathcal{T}(M),  \quad (X,Y,Z) \mapsto R(X,Y)Z: \\
 & R(X,Y)Z = \nabla_X \nabla_Y Z - \nabla_Y \nabla _X Z - \nabla_{[X,Y]}Z,
\end{align*}
where $[X,Y]$ is the Lie bracket. 
In the case that $(M^n,g)$ has constant sectional curvature $S_0 \in \R$, then by \cite[Chapter 4, Lemma 3.4]{doCarmo} we find that
 \begin{equation}
  \label{eq:RinSpaceForms}
  R(X,Y)Z = S_0 (\langle Y,Z\rangle_g X -\langle X,Z\rangle_g Y).
 \end{equation}

\subsection{The hyperbolic plane}\label{subs:hyp}

The manifold we consider is the hyperbolic half-plane, i.e. the set 
$\Hyp^2 = \{(y_1, y_2) \in \R^2 : y_2  > 0\}$ with global coordinates $(y_1,y_2) \mapsto (y_1, y_2)$ and metric
 \[g_{(y_1,y_2)} = \frac{1}{y_2^2 } \begin{pmatrix}
                                  1&0\\0&1
                                 \end{pmatrix}.
\]
It is well known that $(\Hyp^2,g)$ has constant sectional curvature equal to $-1$. The Christoffel symbols of $(\Hyp^2,g)$ are given by the following expressions
$$\Gamma_{11}^1 = \Gamma_{22}^1 = 0, \quad \Gamma_{12}^1 = \Gamma_{21}^1 = - \frac{1}{y_2},\quad 
\Gamma_{11}^2 = \frac{1}{y_2}, \quad  \Gamma_{22}^2 = -\frac{1}{y_2} \mbox{ and } \Gamma_{12}^2 = \Gamma_{21}^2 = 0.$$
One easily verifies that identifying $\partial_{y_1}$ with $(1,0)^t$ and $\partial_{y_2}$ with $(0,1)^t$ we have the following formula for the covariant derivative in $\Hyp^2$ (see \eqref{eq:CovDerLocally})
\begin{equation}
 \label{eq:NablaLocalH}
 \nabla _{\dot \gamma} X = \begin{pmatrix}
\partial_t X_1 - \frac{1}{\gamma_2} (X_1 \partial_t \gamma_2  + X_2 \partial_t \gamma_1)\\
\partial_t X_2 + \frac{1}{\gamma_2} (X_1 \partial_t \gamma_1  - X_2 \partial_t \gamma_2) \end{pmatrix}.
\end{equation}
Each M\"obius transformation that maps $\Hyp^2$ surjectively to $\Hyp^2$ is an isometry. Examples of such tranforsmations are translations in the $(1,0)^t$-direction and dilatations. The geodesics in $\Hyp^2$ are half-circles or generalised half-circles (that is half-lines) centred at a point $(p,0)^t \in \Hyp^2$. Since there are no closed geodesics, $\E(f) >0$ for any closed curve $f$.

\begin{remark}\label{rem:stern}
The geodesic distance between $(x_1,y_1)$, $(x_2,y_2) \in\Hyp^2$ can be expressed as follows
\[ \operatorname{dist}_{\Hyp^2} ((x_1,y_1),(x_2,y_2))= \arccosh\Big(1+ \frac{(x_2-x_1)^2+(y_2-y_1)^2}{2y_2y_1} \Big)\,.\]
In particular, if $x_1=x_2=x$, then $\operatorname{dist}_{\Hyp^2} ((x,y_1),(x,y_2))= \Big|\log \frac{y_2}{y_1} \Big|$. By these formulas it is immediate to see that a ball in $\Hyp^2$ coincides with an Euclidean ball. Indeed,
$$ B_{r}^{\Hyp^2} (x_0,y_0)=\{ (x,y)^t: \; \mbox{dist}_{\Hyp^2} ((x,y),(x_0,y_0)) < r\}= B_{\cosh(r) y_0}^{\R^2} (x_0,y_0 \cosh(r))\, . $$
\end{remark}

Let $f \colon \mathbb{S}^1 \to \Hyp^2$ be a smooth immersion of a closed curve of length $L$. 
We consider several times the following charts on $ \mathbb{S}^1$:
\begin{align}\nonumber
&\phi_i:I \to (\mathbb{S}^1,\diff s), \mbox{ for }i=1,..,4, \mbox{ and }I=(0,L/2), \mbox{ isometries such that}\\ \label{eq:charts}
&\phi_1 (I) = \mathbb{S}^1 \cap \{ (x,y)^t: x>0 \}=V_1, \quad \phi_2(I) = \mathbb{S}^1 \cap \{ (x,y)^t: y>0 \}=V_2, \\ \nonumber
&\phi_3 (I) = \mathbb{S}^1 \cap \{ (x,y)^t: x<0 \}=V_3, \quad \phi_4(I) = \mathbb{S}^1 \cap \{ (x,y)^t: y<0 \}=V_4 \, .
\end{align}
From \eqref{eq:NablaLocalH} we find that the curvature of $f \colon \mathbb{S}^1 \to \Hyp^2$, $f=(f_1,f_2)^{t}$, where we do not raise the indices, 
is given by
\begin{equation}
 \label{eq:curvature}
 \Kapp = \begin{pmatrix}
\partial_s^2 f_1 - \frac{2}{f_2} \partial_s f_1  \partial_s f_2\\
\partial_s^2 f_2 + \frac{1}{f_2} ( (\partial_s f_1)^2  -(\partial_s f_2)^2) \end{pmatrix}.
\end{equation}

\begin{remark}\label{rem:scaling}
Since dilatations will play a crucial role in the subconvergence result we study here shortly the behaviour of the geometrical quantities with respect to these isometries. Let $f \colon \mathbb{S}^1 \to \Hyp^2$ be a smooth immersion of a closed curve and $\tilde{f}\colon \mathbb{S}^1 \to \Hyp^2$ be its rescaling by a factor $r>0$, that is, $\tilde{f}= r f$. Then 
\[ | \partial_x \tilde{f} |^2_{g(\tilde{f})} = \frac{| \partial_x \tilde{f} |^2_{\text{euc}}}{(\tilde{f}_2)^2} = \frac{ r^2| \partial_x f |^2_{\text{euc}}}{ r^2f_2^2}= | \partial_x f |^2_{g(f)} \, ,\]
so that $\diff s_{\tilde{f}} = \diff s_{f}$, and $\partial_{s_{\tilde{f}} } = \partial_{s_{f}}.$ 
It follows that $\partial_{s_{\tilde{f}} } \tilde{f}= r\partial_{s_{f} } f$ and from \eqref{eq:NablaLocalH}
\begin{equation*}
 \nabla _{\partial_{s_{\tilde{f}} }} \tilde{X} = \begin{pmatrix}
\partial_{s_{\tilde{f}} } \tilde{X}_1 - \frac{1}{\tilde{f}_2} (\tilde{X}_1 \partial_{s_{\tilde{f}} } \tilde{f}_2 + \tilde{X}_2\partial_{s_{\tilde{f}} } \tilde{f}_1)\\
\partial_{s_{\tilde{f}} } \tilde{X}_2 + \frac{1}{\tilde{f}_2} (\tilde{X}_1 \partial_{s_{\tilde{f}} } \tilde{f}_1  - \tilde{X}_2 \partial_{s_{\tilde{f}} } \tilde{f}_2) \end{pmatrix} = r \nabla _{\partial_{s_{f} }} X,
\end{equation*}
for $\tilde{X}= r X$ and $X$ a vector field along $f$. 
We also have  $\Kapp_{\tilde{f}}= r\Kapp_{f}$, 
\[ |\Kapp_{\tilde{f}}|^2_{g(\tilde{f})} = \frac{1}{r^2 f_2^2} |\Kapp_{\tilde{f}}|^2_{\text{euc}} = |\Kapp_{f}|^2_{g(f)} \quad \mbox{ and  } \quad | \nabla _{\partial_{s_{\tilde{f}} }}^m \Kapp_{\tilde{f}}|^2_{g(\tilde{f})} = | \nabla _{\partial_{s_{f} }}^m \Kapp_{f}|^2_{g(f)} \, . \]
In particular 
\[ \E(\tilde{f}) = \frac{1}{2} \int_{\mathbb{S}^1} |\Kapp_{\tilde{f}}|^2 _{g(\tilde{f})} \diff s_{\tilde{f}} = 
\frac{1}{2} \int_{\mathbb{S}^1} |\Kapp_{f}|^2 _{g(f)} \diff s_{f} = \E(f)  \, . \]
\end{remark}

An important result on closed curve in the Euclidean space is the Theorem of Fenchel that says that the total curvature of a closed curve is bounded from below by $2 \pi$. The same result is true also in the hyperbolic plane. 

\begin{mythm}[see \cite{Tsukamoto,Szenthe}]\label{Thm:lengthbelow}
The total absolute curvature of a smooth closed curve in a hyperbolic space is at least $2 \pi$. 
\end{mythm}

\subsubsection{Relation with Willmore surfaces of revolution}
\label{subs:revol}
It was already observed by  \cite{LangerSinger1,LangerSinger2} and goes back to 
U. Pinkall and R. Bryant, P. Griffiths  \cite{BrGr} that there is an interesting relation between elastica and Willmore surfaces of revolution that we now shortly review. Let $\gamma: \mathbb{S}^1 \to \mathbb{R}^2_{+}:=\{(x,z)^t: z>0 \}$ be a closed curve parametrised by arc-length. By rotating the curve around the $x$-axis we obtain a surface of revolution in $\R^3$
\[ h_{\gamma}: \mathbb{S}^1 \times [0,2 \pi] \ni (x,\varphi) \mapsto (\gamma_1(x), \gamma_2(x) \cos(\varphi), \gamma_2(x) \sin(\varphi) )^t \in \mathbb{R}^3\, . \]
The induced area element is $\gamma_2(x) \, \diff x \diff \varphi$, the principal curvatures are
\[ \lambda_1= \gamma_1''(x) \gamma_2'(x)- \gamma_2''(x) \gamma_1'(x) \mbox{ and } \lambda_2= \frac{\gamma_1'(x)}{ \gamma_2(x)} \, ,\]
(see \cite[page 161]{doCarmo1}) and the Willmore energy of this surface of revolution is
given by
\[ W(h_{\gamma})= \int H^2 \diff S =  \frac{\pi}{2} \int_{\mathbb{S}^1} \left( \gamma_1''(x) \gamma_2'(x)- \gamma_2''(x) \gamma_1'(x) +\frac{\gamma_1'(x)}{ \gamma_2(x)} \right)^2 \gamma_2(x)\; \diff x \, , \]
where $H$ denotes the mean curvature, that is $H=\frac12 (\lambda_1+\lambda_2)$. 

Now we consider the same curve as a curve $\gamma: \mathbb{S}^1 \to \Hyp^2$. Being this curve parametrised in Euclidean arc-length we find
$\partial_s \gamma (x) = \frac{1}{|\partial_x \gamma(x)|_g} \partial_x \gamma (x) = \gamma_2(x) \gamma' (x)$, $\partial_s =  \gamma_2(x)  \partial_x$  and  for the hyperbolic curvature from \eqref{eq:curvature} with some elementary computations that
$|\Kapp|_{g}^2 =  \gamma_2^2 \Big[( \gamma_1'' \gamma_2' - \gamma_2'' \gamma_1' +\frac{\gamma_1'}{\gamma_2})^2 + 4 \gamma_2'' \frac{1}{\gamma_2} \Big]$.
It follows that 
\begin{align*}
\E(\gamma) & = \int_{\mathbb{S}^1}   |\Kapp|_{g}^2(x)\frac{1}{\gamma_2(x)} \diff x =  \frac{2}{\pi} W(h_{\gamma}) + 4 \int_{\mathbb{S}^1} \gamma_2''(x) \,  \diff x = \frac{2}{\pi} W(h_{\gamma}) \, ,
\end{align*}
since there is no boundary. Since the Willmore energy is invariant under rescaling, it is then not surprising that the same holds for the elastic energy of curves in $\Hyp^2$, see Remark \ref{rem:scaling}.
\\
Note that, even though the energies coincide, the elastic flow does not fully describe the Willmore flow. Indeed, let $f$ be the global solution to \eqref{eq:stat} from \autoref{mainThm}. Rotating the family $f$ around the $x$-axis gives a global family $h_f$ 
of smooth surfaces of revolution that satisfies $
\frac{\diff}{\diff t} W(h_f) = \frac{\diff}{\diff t} \frac{\pi}{2} \mathcal{E}(f) = 
-\|\partial_t f\|_{L^2(\Sph^1;\Hyp^2)}^2$. Thus $h_f$ decreases the Willmore energy of the initial surface of revolution $h_{f_0}$. Moreover, the family has the property that each $h_f$ is a surface of revolution, similar to the evolution under the Willmore flow (see \cite[Section 2]{Blatt}). Nevertheless 
a variation $f + t\psi$ of a fixed curve $f$ only corresponds to rotational invariant variations $h_f + t h_\psi$ of $h_f$, resulting in a gradient 
with respect to a closed subspace of $L^2(\Sigma;\R^3)$ only. 
 
\subsection{Evolution equations}

Let $f \colon \mathbb{S}^1 \to (M,g)$ be a smooth immersion of a closed curve. For convenience we use the following notation
\[\nabla_{\partial_x} = \nabla_{\partial_x f} \mbox{ and } \nabla_{\partial_s} = \nabla_{\partial_s f} \]
where $\partial_s f = \frac{1}{|\partial_xf|_g}\partial_xf \in \mathcal{T}(f)$. For $V \in \mathcal{T}(f)$, $V^{\bot}$ denotes the projection onto the subspace orthogonal to $\partial_s f$. In particular,
\begin{equation}
 \label{eq:defNabla}
 \quad \nabla_{\partial_s}^\bot \cdot = \nabla_{\partial_s f}\cdot - \langle \nabla_{\partial_s f}\cdot, \partial_s f\rangle_g \partial_s f.
\end{equation}
Similarly, if $f \colon \mathbb{S}^1 \times (0,T) \to (M,g)$ for some $T>0$, where we equip $(0,T)$ with the coordinate $t$, we set
\[
 \nabla_{\partial_t} = \nabla_{\partial_t f},\quad \nabla_{\partial_t}^\bot \cdot = \nabla_{\partial_t f}\cdot - \langle \nabla_{\partial_tf}\cdot, \partial_s f\rangle_g \partial_s f.\]

Our aim now is to compute the evolution equations satisfied by derivatives of the curvature of any solution of \eqref{eq:flow}. In order to do that we have also to derive the evolution equations satisfied by other geometric quantities. We give here the results and postpone the quite technical proofs to Appendix \ref{sec:tech}. For completeness we only note here that in the computations we repeatedly use \eqref{eq:compcurves} as follows:
 For two vector fields $X,Y$ along $f \colon \mathbb{S}^1 \times (0,T) \to (M,g)$, $T>0$, we have
\begin{equation}
 \langle X,\nabla_{\partial_t} Y\rangle_g = \partial_t \langle X,Y\rangle_g -   \langle \nabla_{\partial_t}X, Y\rangle_g  \mbox{ and }   \langle X,\nabla_{\partial_s} Y\rangle_g = \partial_s \langle X,Y\rangle_g -   \langle \nabla_{\partial_s}X, Y\rangle_g \, .
\label{eq:PartialIntegration}
\end{equation}

\begin{lemma}\label{lemma:evo}
Let $(M,g)$ be a smooth Riemannian manifold with constant sectional curvature $S_0$. Let $T>0$ and $f \colon \mathbb{S}^1 \times (0,T) \to (M,g)$ be smooth. Let $\partial_t f = V + \phi \partial_s f$ where $\langle V, \partial_s f\rangle = 0$ and whence $\phi = \langle \partial_sf, \partial_t f\rangle$. Then the following evolution formulas are satisfied on $\mathbb{S}^1 \times (0,T)$
\begin{align}
 \nabla_{\partial_t} \partial_x f & = \nabla _{\partial_x}\partial_t f, \label{eq:commutingcoordinates}\\
 \partial_t (|\partial_x f|_g) &= (\partial_s \phi - \langle V,\Kapp\rangle_g )|\partial_xf|_g, \label{eq:evolutionoflineelement1}\\
 \partial_t (\diff s) &=  (\partial_s \phi - \langle V,\Kapp\rangle_g )\diff s, \label{eq:evolutionoflineelement2}\\
 \nabla_{\partial_t}\partial_s f - \nabla_{\partial_s}\partial_t f &= (\langle V,\Kapp\rangle_g -\partial_s \phi)\partial_s f
 \label{eq:commutator1}.
\end{align}
For any  vector field $\Phi \colon I \times [0,T) \to TM$ 
and for any  vector field $N \colon \mathbb{S}^1 \times [0,T) \to TM$ normal to $f$  (i.e. $\langle N, \partial_s f\rangle _g = 0$) we have
\begin{align}
\nabla_{\partial_s} N &= \nabla_{\partial_s}^\bot N - \langle N, \Kapp\rangle _g \partial_s f,\label{eq:NormalDerOfNormal}\\
 \nabla_{\partial_t} \partial_s f &= \nabla_{\partial_s}^\bot V + \phi \Kapp.\label{eq:EvoUnitVelocity}\\
 \nabla_{\partial_t}N &= \nabla_{\partial_t}^\bot N - \langle N, \nabla_{\partial_s}^\bot V + \phi \Kapp\rangle \partial_s f
 \label{eq:EvoNormal}, \\
\nabla_{\partial_t}\nabla_{\partial_x}\Phi - \nabla_{\partial_x}\nabla_{\partial_t}\Phi &= S_0(\langle{\partial_xf,\Phi}\rangle_g V- \langle V, \Phi\rangle_g \partial_xf),\label{eq:EvoSCommutators1}\\
\nabla_{\partial_t}\nabla_{\partial_s}\Phi - \nabla_{\partial_s}\nabla_{\partial_t}\Phi &=- (\partial_s\phi - \langle V,\Kapp\rangle _g)\nabla_{\partial_s} \Phi + S_0(\langle{\partial_sf,\Phi}\rangle_g V- \langle V, \Phi\rangle_g \partial_s f),
\label{eq:commutatorS}\\
\nabla_{\partial_t}^\bot \nabla_{\partial_s}^\bot N - \nabla_{\partial_s}^\bot \nabla_{\partial_t}^\bot N
& =
(\langle V,\Kapp\rangle_g - \partial_s \phi) \nabla_{\partial_s}^\bot N + \langle N,\Kapp\rangle_g \nabla_{\partial_s}^\bot V - \langle N, \nabla_{\partial_s}^\bot V \rangle \Kapp \, .\label{eq:Commutator_Normal_Normal_S}
\end{align}
In particular,
\begin{align}
\nabla_{\partial_t}\Kapp &= (\nabla_{\partial_s}^\bot)^2 V  - 
\langle{\nabla_{\partial_s}^\bot V,\Kapp}\rangle_g \partial_s f + \phi 
\nabla_{\partial_s} \Kapp + \langle V,\Kapp\rangle_g \Kapp +S_0 V,
\label{eq:EvoKappa}\\
\nabla_{\partial_t}^\bot\Kapp &= (\nabla_{\partial_s}^\bot)^2 V   + \phi 
\nabla_{\partial_s} \Kapp + \langle V,\Kapp\rangle_g \Kapp + S_0 V.\label{eq:EvoNormalKappa}
\end{align}
\end{lemma}

 The proof is given in Appendix \ref{sec:tech}. The same formulas in the case $M=\mathbb{R}^n$ with the standard metric has been computed in \cite[Lemma 2.1]{DKS}.

\begin{remark}\label{rmk:proofGrad}
 With the formulas just derived we can verify that $\nabla_{L^2}\E_\lambda(f)$ is given as in \eqref{eq:gradientElambda} if $M$ has constant sectional curvature $S_0$. Let $T>0$ and $f \colon \mathbb{S}^1 \times (0,T) \to (M^n,g)$ be smooth. We write $\partial_t f = V + \phi \partial_s f$, where $\langle V, \partial_s f\rangle = 0$. Then we find using \eqref{eq:EvoNormalKappa}, \eqref{eq:evolutionoflineelement2} and direct computation that
 \begin{align*} 
  \frac{\diff}{\diff t} \E_{\lambda}(f) 
   &= \int_{\mathbb{S}^1} \langle \Kapp ,   (\nabla_{\partial_s}^\bot)^2 V   + \phi \nabla_{\partial_s} \Kapp 
   + \langle V,\Kapp\rangle_g \Kapp + S_0 V \rangle _g \diff s \\ 
	& \qquad + \int_{\mathbb{S}^1} (\frac{1}{2}|\Kapp|_g^2 + \lambda )(\partial_s \phi-\langle V,\Kapp \rangle _g )\diff s \, .
	\end{align*}
	Integrating by parts
  \allowdisplaybreaks{\begin{align*}
	\frac{\diff}{\diff t} \E_{\lambda}(f)  
   &= 
    \int_{\mathbb{S}^1} \langle 
   (\nabla_{\partial_s}^\bot)^2 \Kapp + \frac{1}{2} |\Kapp|^2_g \Kapp + S_0\Kapp- \lambda \Kapp
   ,V\rangle _g \diff s  = \langle \nabla_{L^2}\E_\lambda(f), \partial_t f \rangle_{L^2} \, ,
 \end{align*}}
and hence \eqref{eq:gradientElambda} follows. From this computation we see in particular that if we consider the steepest descent flow $\partial_t f = - \nabla_{L^2}\E_\lambda(f)$, that is $V=- \nabla_{L^2}\E_\lambda(f) $ and $\phi=0$, then 
\begin{equation}\label{eq:endecrease}
\frac{\diff}{\diff t} \E_{\lambda}(f) = - \int_{\mathbb{S}^1} | \nabla_{L^2}\E_\lambda(f)|_g^2 \diff s  \leq 0 \,.
\end{equation}
\end{remark}

In order to give the evolution equations satisfied by the derivatives of the curvature we need to introduce first some notation. Similar to \cite[Lem.2.3]{DKS} and \cite[Sec.3]{AnnaPaola1} we denote by the product $ N _1 *  N _2 * \cdots * N _k$ of normal vector fields $ N _1, \ldots ,  N _k$ the function $\langle  N _1,  N _2\rangle_g \cdots \langle  N _{k-1},  N _k\rangle_g$ if $k$ is even, and the vector field
$\langle  N _1,  N _2\rangle_g \cdots \langle  N _{k-2},  N _{k-1}\rangle_g  N _k$ if $k$ is odd. Furthermore we denote the pointwise product by functions again by $*$, if some of the $N_i$ are functions. Moreover, we denote by $P_b^{a,c}( N )$ any linear combination of terms of the type
\[
 (\nabla_{\partial_s}^\bot)^{i_1}  N  * \cdots * (\nabla_{\partial_s}^\bot)^{i_b}  N  \text{ with }i_1 + \ldots + i_b = a \text{ and }\max i_j \leq c
\]
with universal, constant coefficients. Usually we have $N = \Kapp$. Notice that $a$ gives the total number of derivatives, $b$ denotes the number of factors and $c$ gives a bound on the highest number of derivatives falling on one factor. We observe the two properties 
\[P_b^{a,c}( N ) * P_\beta^{\alpha,\gamma}( N ) = P_{b+\beta}^{a+\alpha, \max\{c,\gamma\}}( N ) \text{ and } \nabla_{\partial_s}^\bot P_b^{a,c}( N ) = P_b^{a+1,c+1}( N ),\]
where we abuse the notation $\nabla_{\partial_s}^\bot P^{a,c}_b (N)$ for $\partial_s P^{a,c}_b (N)$ if $b$ is even. 
Adopting this notation we find from \autoref{eq:flow} that
\begin{equation}
\label{eq:GradAsPolynom} 
 - \nabla_{L^2}\E_\lambda(f) = -(\nabla_{\partial_s}^\bot )^2 \Kapp+ P_3^{0,0}(\Kapp)  + P_1^{0,0}(\Kapp),
\end{equation}
where we do not keep track of the constants $\lambda$ and $S_0$ since they are fixed. 
Finally we derive the evolution equations satisfied by the derivative of the curvature. 

\begin{lemma}\label{lem:dercurvature}
 Under the assumption of \autoref{lemma:evo} 
we have for any $m \in \N_0$
 \begin{align*}
  \nabla_{\partial_t}^\bot (\nabla_{\partial_s}^\bot)^m\Kapp 
  &= -(\nabla_{\partial_s}^\bot)^{4+m} \Kapp + P_3^{2+m, 2+m}(\Kapp) + 
	P_1^{2+m,2+m} (\Kapp) 
	\\ &\quad  + P_5^{m,m}(\Kapp) + 
	P_3^{m,m} (\Kapp) + 
	P_1 ^{m,m} (\Kapp) 
	.
 \end{align*}
\end{lemma}

The proof is given in Appendix \ref{sec:tech} and the result in the case $M=\mathbb{R}^n$ with the standard metric has been given in \cite[Lemma 2.3]{DKS}. We have just derived the evolution equations of the normal component of the derivatives of the curvature. In order to get control of the flow we need information on the whole derivative. For this reason in the next lemma we look at the relation between $\nabla_{\partial_s}^m \Kapp$ and $(\nabla_{\partial_s }^\bot)^m\Kapp$ for $m \in \N$.

\begin{lemma}\label{lem:dercurvnormnicht}
 Under the assumption of \autoref{lemma:evo} then
 $\nabla_{\partial_s}\Kapp = \nabla_{\partial_s }^\bot\Kapp - |\Kapp|^2 \partial_s f$, 
 and for $m \geq 2:$
 \[
  \nabla_{\partial_s}^m \Kapp = (\nabla_{\partial_s }^\bot)^m\Kapp 
  + \sum_{b=2,\; b\text{ even}}^{m+1} P_b^{m+1 -b, m+1-b}(\Kapp) \partial_sf 
  +  \sum_{b=3,\; b\text{ odd}}^{m+1} P_b^{m+1 -b, m+1-b}(\Kapp).
 \]
\end{lemma}

Also in this case the proof is given in Appendix \ref{sec:tech} and the result in $\R^n$ is given in \cite[Lemma 2.6]{DKS}. 

\begin{lemma}\label{lem:controlparsuff}
Under the assumption of \autoref{lemma:evo} for any  vector field $N \colon \mathbb{S}^1 \times (0,T) \to TM$ normal to $f$  (i.e. $\langle N, \partial_s f\rangle _g = 0$) we have for any $m \in \N$ 
$$ \nabla_{\partial_x}^m N = \gamma^m   \nabla_{\partial_s}^m N + \sum_{j=1}^{m-1} P_{m,j}(\gamma, ....,\partial_x^{m-j} \gamma ) \nabla_{\partial_s}^j N\, , $$
with $\gamma = |\partial_x f|_g$ and $P_{m,j}$ polynomials of degree at most $m-1$.
\end{lemma} 

The following lemma gives the main tool to derive from the evolution equations in $M$ of the curvature and its derivative to a differential equation for their $L^2$-norms. 

\begin{lemma}[c.f. {\cite[Lemma 2.2]{DKS}}]\label{lem:integration}
 Let $f \colon \mathbb{S}^1 \times (0,T) \to (M^n,g)$ be a family of curves such that  $\partial_t f = V$, where $V$ is a vector field normal to $f$. Then for any smooth normal vector field $N$ along $f$ satisfying
 \begin{equation}
 \label{eq:DefY}
 \nabla_{\partial_t}^\bot N + (\nabla_{\partial_s}^\bot)^4 N = Y,
 \end{equation}
we find
 \begin{equation}
  \frac{\diff}{\diff t} \frac{1}{2} \int_{\Sph^1} |N|^2_g \diff s + \int_{\Sph^1} |(\nabla_{\partial_s} ^\bot)^2 N|_g^2 \diff s
  = \int_{\Sph^1} \langle Y,N\rangle _g \diff s - \frac{1}{2} \int_{\Sph^1} \langle V, \Kapp\rangle _g |N|_g^2 \diff s.
 \end{equation}
\end{lemma}

\begin{proof}
 Taking the scalar product of \eqref{eq:DefY} with $N$ and integrating we obtain
 \[
   \int_{\Sph^1} \langle N, \nabla_{\partial_t}^\bot N\rangle _g \diff s + \int_{\Sph^1} \langle N, (\nabla_{\partial_s}^\bot)^4N \rangle \diff s  =\int_{\Sph^1} \langle N,  Y\rangle_g \diff s.
 \]
 Since $N$ is normal, using \eqref{eq:PartialIntegration} and \eqref{eq:evolutionoflineelement2} (with $\phi=0$) we get
  \begin{equation}
     2\int_{\Sph^1} \langle N, \nabla_{\partial_t}^\bot N\rangle _g \diff s_f 
  =  \int_{\Sph^1}(\partial_t|N|_g^2) \diff s_f
  =  \frac{\diff}{\diff t}  \int_{\Sph^1}|N|_g^2 \diff s_f + \int_{\Sph^1} \langle V, \Kapp\rangle _g |N|_g^2 \diff s_f. \label{eq:Lemma2.2_1}
  \end{equation}
Similarly, using again  \eqref{eq:PartialIntegration}  and that $\mathbb{S}^1$ has no boundary we find
\begin{align}
  \int_{\Sph^1} \langle N, (\nabla_{\partial_s}^\bot)^4 N \rangle_g \diff s_f 
  &= -  \int_{\Sph^1} \langle \nabla_{\partial_s} N,  (\nabla_{\partial_s}^\bot)^3 N\rangle_g \diff s_f\nonumber 
  =  \int_{\Sph^1} \langle (\nabla_{\partial_s}^\bot)^2 N,  (\nabla_{\partial_s}^\bot)^2 N \rangle_g \diff s_f \nonumber\\
  &=  \int_{\Sph^1} |(\nabla_{\partial_s} ^\bot)^2 N|_g^2 \diff s_f ,
  \label{eq:Lemma2.2_2}
\end{align}
and the claim follows from \eqref{eq:Lemma2.2_1} and \eqref{eq:Lemma2.2_2}.
\end{proof}

\section{Short time existence}
\label{section:STE}

In this section we give a complete proof to the short time existence of the elastic flow in the hyperbolic plane. Thus we have $S_0 = -1$ in the following.

\begin{mythm}\label{thm:STE}
Let $f_0\colon \mathbb{S}^1 \to \mathbb{H}^2$ be an immersion. 
\begin{enumerate}[(i)]
 \item If $f_0$ is smooth there exists some $T > 0$ and a smooth immersed solution $f\colon \mathbb{S}^1 \times [0,T] \to \mathbb{H}^2$ to the elastic flow
 \begin{equation}\label{eq:EvolEqn}
   \begin{cases}
           \partial_t f = -(\nabla_{\partial_s}^\bot)^2 \Kapp - \frac{1}{2} |\Kapp|^2_g \Kapp +(1+ \lambda) \Kapp, & \text{ in } \mathbb{S}^1 \times [0,T],\\
           f(0,x)  = f_0(x), & \text{ on } \mathbb{S}^1.
          \end{cases}
 \end{equation}
 The solution is unique up to reparametrisations.
 \item If $f_0 \in \CalC^{5,\alpha}$, then there exists a solution $f$ to \eqref{eq:EvolEqn} such that $f$ and $\partial_t f$ lie in the parabolic H\"older space $\HS[1+\alpha]{\frac{1+\alpha}{4}}(\Sph^1 \times [0,T];\mathbb{H}^2)$. The solution is smooth on $(0,T]\times \Sph ^1$. 
\end{enumerate}
\end{mythm}


The proof consists of several steps. First we reduce the equation to a quasilinear parabolic equation (\eqref{eq:reducedEq} below). 
By Schauder estimates for the linearised equation at the initial value $f_0$ we can solve the nonlinear equation with a fixed point method (\autoref{thm:nonlinear_schauder}).
Bootstrapping then yields the smoothing effect (\autoref{thm:parabolic_smoothing} and \autoref{corol:smoothness_soln}). The uniqueness result for the quasilinear equation follows from the fixed-point method and time-uniform estimates (see Paragraph \ref{proof:uniqueness} in the appendix). \\
Note that a proof of the short time existence for open curves in the context of Sobolev spaces has been shown in \cite{STE_LP}.

\subsection{Hanzawa-type transformation and the solution of an equivalent PDE}

To solve the geometric PDE from \autoref{eq:EvolEqn} we will first write the initial value $f_0$ as a normal graph over some smooth curve (see \autoref{prop:NormalGraph}) and then observe how we can transform the PDE into a quasilinear parabolic equation for some unknown $u \colon \Sph^1 \times [0,T]\to \R$. For $\Hyp^2$ we will repeatedly use the global coordinate chart  from \autoref{subs:hyp}.

\begin{proof}[{Proof of \autoref{thm:STE}}] 1) 
We start with the  existence and smoothness of the solution for smooth initial values $f_0$, but lay the foundations for the proof of the existence and uniqueness with lower regularity of the initial value. Thus we let $f_0 \colon \Sph^1 \to \Hyp^2$ be the smooth immersed initial value of \eqref{eq:EvolEqn} and let $\overline f\colon \Sph^1 \to \Hyp^2$ be a smooth curve with normal unit vector field $\overline N$ along $\overline f$. Using the global chart of $\Hyp^2$ 
and identifying $T_y\R^2 \cong \R^2$ we can translate $\overline N(x) \in T_{\overline f(x)}\Hyp^2$ to $N(x) \in T_{f_0(x)}\Hyp^2$, a vector field along $f_0$ (which is \emph{not} the parallel transport of $\overline N$). Moreover, for any vector field $\Phi$ along $f_0$ we denote the tangential and normal projection along $f_0$ by
\[
 \Pi_{f_0}^\top \Phi \defeq \langle \Phi ,\partial_sf_0\rangle_g \partial_s f_0, \quad \Pi_{f_0}^\bot \Phi \defeq \Phi-\Pi_{f_0}^\top \Phi. \]
With this notation we have $\nabla_{\partial_s}^\bot = \Pi_{f_0}^\bot \circ\nabla_{\partial_s}$ (c.f. \autoref{eq:defNabla}). Analogously we define $\Pi^\bot_h$ for any $\CalC^1$-immersion $h$. 
We have now introduced the notation to state the following proposition, whose proof is given in Appendix \ref{proof:lemma_normal_graph}.

\begin{prop}\label{prop:NormalGraph}
 Let $m\in \N_0$ and $f_0 \colon \Sph^1 \to \Hyp^2$ be a $\CalC^{4+m,\alpha}$-immersion. Then there exists some $\mu> 0$ and a smooth, immersed reference curve $\overline f\colon \Sph^1 \to \Hyp^2$ with smooth unit normal vector field $\overline N$ along $\overline f$ such that for all $h\colon \Sph^1 \to \Hyp^2$ with $\|h - f_0\|_{\CalC^{4+m,\alpha}} \leq \mu$ we have:
 \begin{enumerate}[(i)]
  \item The translated vector field $\overline N$ along $h$ satisfies $\Pi_{h(x)}^\bot (\overline N(x)) \neq 0$ for all $x \in \Sph^1$, whence $\Pi_h^\bot (\overline N)$ is a basis for the normal bundle of $h$ and abusing the notation we find that the mapping $\Pi_h^\bot$ is an isomorphism when restricted to a mapping from the normal bundle of $\overline f$ to the normal bundle of $h$.
  \item There exists some reparametrisation of $h$ such that 
  $h = \overline f + u \overline N$ for some unique $\CalC^{4+m,\alpha}$-function $u\colon \Sph^1 \to \R$.
 \end{enumerate}
\end{prop}

To show existence we fix some $\overline f$ as in \autoref{prop:NormalGraph}. We have
\begin{equation}
 \label{eq:u_0} f_0 = \overline f + u_0 \overline N
\end{equation}
for some smooth function $u_0$. 
To find a solution to \autoref{eq:EvolEqn} we make the ansatz
\begin{equation}
\label{eq:ansatz_f}
f(x,t) = \overline f(x) + u(x,t) \overline N(x)
\end{equation}
and calculate, writing $f = (f_1, f_2)$ in our chart and $|\,\cdot\,|_e$ for the Euclidean norm 
the following expressions for $f$ depending on $u$ and its derivatives:
\begin{align*}
\partial_s f = \frac{\partial_xf }{|\partial_xf|_{g(f)}}&= \frac{\partial_x \overline f + (\partial_x u)\overline N + u \partial_x\overline  N}{|\partial_xf|_e \frac{1}{f_2}}
= \frac{(\partial_x u)\overline N}{|\partial_xf|_e }f_2 + P_1(\cdot,u,|\partial_xf|_e^{-1})
\end{align*}
for a smooth function $P_1 \colon \Sph^1\times \R^2 \to T\R^2$ which is a polynomial in the latter arguments for fixed $x \in \Sph^1$. The coefficients of this polynomial are smooth in $x$ as they depend only on $\overline f$ and $\overline N$. 
From \eqref{eq:NablaLocalH} we find for vector fields $\Phi$ along $f$ that
\begin{align*}
 \nabla_{\partial_s} \Phi
&= 
\begin{pmatrix}
                 \frac{f_2}{|\partial_xf|_e } \partial_x \Phi_1 - \frac{1}{|\partial_x f|_e}(\Phi_1 \partial_x f_2 + \Phi_2 \partial_x f_1)\\
                 \frac{f_2}{|\partial_xf|_e } \partial_x \Phi_2 + \frac{1}{|\partial_x f|_e}(\Phi_1 \partial_x f_1 - \Phi_2 \partial_x f_2)
                 \end{pmatrix} .
\end{align*}
Whence, since $\Kapp$ is already normal to $f$,
\begin{align*}
 \Kapp &= \nabla_{\partial_s} \partial_s f = \nabla_{\partial_s}^\bot \partial_sf = \Pi_f^\bot \circ \nabla_{\partial_s} \partial_sf = \Pi_f^\bot (\partial_s \partial_s f + P_2(\cdot,f,\partial_s f))\\
  &= \Pi_f^\bot \left(\frac{(\partial_x^2 u)\overline N }{|\partial_xf|^2_e}f_2^2 + P_3(\cdot,u,\partial_x u,|\partial_xf|_e^{-1})\right)= \frac{\partial_x^2 u}{|\partial_xf|^2_e}f_2^2 \Pi_f^\bot\overline  N + P_4(\cdot,u,\partial_x u,|\partial_xf|_e^{-1}),
\end{align*}
where we used that $\Pi^\bot_f$ only contributes terms of $u, \partial_x u$ and $|\partial_xf|^{-1}_e$. This will be used repeatedly in the following. We find
\begin{align*}
 \nabla_{\partial_s}^\bot \Kapp &= \Pi_f^\bot \nabla_{\partial_s} \Kapp = \Pi_f^ \bot (\partial_s \Kapp ) +  P_5(\cdot,u,\partial_x u,\partial_x^2 u,|\partial_xf|_e^{-1})\\
 &= \frac{\partial_x^3 u}{|\partial_xf|^3_e}f_2^3 \Pi_f^\bot\overline  N + P_6(\cdot,u,\partial_x u, \partial_x^2 u,|\partial_xf|_e^{-1}),\\
(\nabla_{\partial_s}^\bot)^2 \Kapp &= \Pi_f^\bot \nabla_{\partial_s}  \nabla_{\partial_s}^\bot \Kapp = \frac{\partial_x^4 u}{|\partial_xf|^4_e}f_2^4 \Pi_f^\bot \overline N + P_7(\cdot,u,\partial_x u, \partial_x^2 u, \partial_x^3 u,|\partial_xf|_e^{-1})
\end{align*}
and thus using \eqref{eq:gradientElambda} we finally find
\[
\nabla_{L^2}\E_\lambda(f) = \frac{\partial_x^4 u}{|\partial_xf|^4_e}f_2^4 \Pi_f^\bot \overline  N + P_8(\cdot,u,\partial_x u, \partial_x^2 u, \partial_x^3 u,|\partial_xf|_e^{-1}),
\]
where  $P_8$ is some smooth function $P_8 \colon \Sph^1 \times \R^5 \to T\R^2$ which is a polynomial for fixed $x \in \Sph^1 $ in the latter arguments. Since $\Pi_f^\bot \overline N$ is nonvanishing for small $t \in [0,T]$ by \autoref{prop:NormalGraph}, and  $\Pi_f^\bot\overline  N$ and $\nabla \E_\lambda(f)$ are both orthogonal to $\partial_s f$ 
we may write
\[
 P_8(\cdot,u,\partial_x u, \partial_x^2 u, \partial_x^3 u,|\partial_xf|_e^{-1}) =- P(\cdot,u,\partial_x u, \partial_x^2 u, \partial_x^3 u,|\partial_xf|_e^{-1}) \Pi_f^\bot\overline  N
\]
for some smooth $P \colon  \Sph^1 \times \R^5 \to \R$. Since $\partial_t f =\dot u N$ we are lead to consider the equation 
\[
 \dot u \Pi_f^\bot\overline  N = \Pi_f^\bot (\dot u\overline  N)
 = \Pi_f^\bot (\partial_t f) 
 =-\nabla\E_\lambda(f) =  \left(- \frac{\partial_x^4 u}{|\partial_xf|^4_e}f_2^4 +P(\cdot,u,\partial_x u,\ldots)\right) \Pi_f^\bot \overline N,
\]
which is equivalent to 
\begin{equation}
\label{eq:reducedEq}
\begin{cases}
 \dot u = - \frac{\partial_x^4 u}{|\partial_xf|^4_e}f_2^4 +P(\cdot,u,\partial_x u, \partial_x^2 u, \partial_x^3 u,|\partial_xf|_e^{-1})& \text{on }\Sph^1 \times [0,T)\\
 u(\cdot,0) =  u_0(\cdot)
 &\text{on }\Sph^1,
\end{cases}
\end{equation}
with initial value $u_0$ from \autoref{eq:u_0}. 
Since $u_0$ is smooth we find from  \autoref{thm:nonlinear_schauder} and \autoref{corol:smoothness_soln} below that there exists a smooth and unique solution $u \colon \Sph^1 \times [0,T] \to \R$ to \eqref{eq:reducedEq}. 
Defining $f$ by \autoref{eq:ansatz_f}, where $u$ is the obtained solution of \eqref{eq:reducedEq}, we can show that an adequate reparametrisation of the function $f$  
solves \eqref{eq:EvolEqn}. {By construction} $f$ satisfies 
\begin{equation}
 \label{eq:projEvol}\Pi_f^\bot (\partial_t f)  =-\nabla \E_\lambda(f).
\end{equation}
Thus for the smooth function $\xi = \langle \partial_t f, \partial_s f\rangle _{g(f)}(|\partial_x f|_{g(f)})^{-1} 
$ we have 
$
\partial_t f =-\nabla\E_\lambda(f) + \xi \partial_x f.$ 
Then there exists a unique smooth solution $\Phi$ to the ODE system
\begin{equation}\label{eq:flowdiffeos}
 \begin{cases}
  \dot \Phi(x,t) = -\xi(\Phi(x,t),t), & (x,t) \in \Sph^1 \times [0,T]\\
  \Phi(x,0)=x,& x\in\Sph^1.\end{cases}
\end{equation}
Thus $(\Phi(\cdot,t))$ is a family of diffeomorphism  of $\Sph^1$ when we again choose some smaller $T>0$, if necessary. For the composition $\tilde f = f \circ \Phi$ we find $\tilde f(\cdot, 0) = f(\operatorname{id}, 0) =  f_0$ and
\[
\partial_t \tilde f = \partial_t f \circ \Phi + (\partial_x f \circ \Phi) \dot \Phi = -\nabla \E_\lambda(f)\circ \Phi + (\xi \partial_x f)\circ \Phi - (\partial_x f \circ \Phi) \xi \circ \Phi= -\nabla \E_\lambda(\tilde f),\]
as $\nabla \E_\lambda(f)\circ \Phi = \nabla\E_\lambda(f\circ \Phi)$ from the invariance of $\E_\lambda$ under diffeomorphisms of $\Sph^1$.
\\2) Uniqueness of the smooth solution. 
Let $f\colon \Sph^1 \times [0,T]$ be any solution to \eqref{eq:EvolEqn}. We will show that $f$ is equal to our constructed solution $\tilde f$ up to a diffeomorphism of $\Sph^1$. Let us again fix some $\overline f$ as in \autoref{prop:NormalGraph}. For $T$ small enough there exists a solution $\Psi \in \CalC^\infty(\Sph^1 \times [0,T],\Sph^1)$ of the ODE
\[
 \begin{cases}
  \dot \Psi(x,t) = \displaystyle -\frac{\langle (\partial_t f )(\Psi(x,t),t), \partial_x \overline f(x)\rangle_{g(\overline f(x))}}{\langle (\partial_x f)(\Psi(x,t),t), \partial_x \overline f (x)\rangle_{g(\overline f(x))} }, & (x,t) \in \Sph^1 \times [0,T],\\
  \Psi(x,0)=x,& x\in\Sph^1\end{cases}
\]
of diffeomorphisms $\Psi(\cdot,t)$ of $\Sph^1$. For the composition $\hat f = f \circ \Psi$ we find $\hat f(x,0) = f(\Psi(x,0),0) = f(x,0) = f_0 = \overline f + u_0 N$ and
\begin{align*}
  \partial_t \Pi_{\overline f}^\top({\hat f} - \overline f)&=   \Pi_{\overline f}^\top(\partial_t{\hat f} ) =\Pi_{\overline f}^\top(\partial_t{f}\circ \Psi + (\partial_x f)\circ \Psi \dot \Psi)\\
  &= \left( \langle\partial_t{f}\circ \Psi,\partial_x \overline f \rangle + \langle (\partial_x f)\circ \Psi, \partial_x \overline f \rangle \dot \Psi \right)\frac{\partial_x \overline f}{|\partial_x \overline f|^2} =0.
\end{align*}
Whence $\hat f$ satisfies \eqref{eq:ansatz_f} for some unique function $u$, which then solves \eqref{eq:reducedEq} and $u(0)=u_0$, whence $\hat f$ equals the constructed solution $\tilde f$, i.e. $f = \tilde f \circ \Psi^{-1}$ is a reparametrisation of our constructed solution.\\
3) Existence of a solution for $f_0 \in \CalC^{5,\alpha}$. For $f_0 \in \CalC^{5,\alpha}$ we apply \autoref{prop:NormalGraph} (ii) and have (after reparametrising $f_0$) the representation \eqref{eq:u_0} with some function $u_0 \in \CalC^{5,\alpha}$. We proceed as before and have \eqref{eq:reducedEq}. 
The Schauder theory from \autoref{thm:nonlinear_schauder} shows that \eqref{eq:reducedEq} has a unique solution $u$ in the parabolic H\"older space $\HS[4+\alpha]{\frac{4+\alpha}{4}}(\Sph^1 \times [0,T])$ for some $T>0$, and $u$ satisfies $\HS[5+\alpha]{\frac{5+\alpha}{4}}(\Sph^1 \times [0,T])$ by  \autoref{corol:smoothness_soln}. Furthermore, we find from \autoref{thm:parabolic_smoothing} that the solution $u$ (and whence $f$) is smooth on $\Sph^1 \times [\delta,T]$ for any $\delta > 0$. 
Then $f \in \HS[5+\alpha]{\frac{5+\alpha}{4}}(\Sph^1 \times [0,T])$ satisfies \eqref{eq:projEvol}, and for the function $\xi$ as above we have $\xi \in \HS[1+\alpha]{\frac{1+\alpha}{4}}(\Sph^1 \times [0,T]) \cap \CalC^\infty( \Sph^1 \times (0,T])$.
The unique solution $\Phi$ to the ODE system
\eqref{eq:flowdiffeos}
satisfies $\Phi \in \HS[1+\alpha]{\frac{1+\alpha}{4}}(\Sph^1 \times [0,T],\Sph^1)\cap \CalC^\infty(\Sph^1 \times (0,T],\Sph^1)$ and $\partial_t \Phi \in \HS[1+\alpha]{\frac{1+\alpha}{4}}(\Sph^1 \times [0,T],\Sph^1)$ by \cite[Theorem 2.5.13]{Gerhardt}. 
Then $\tilde f = f \circ \Phi$ satisfies $\tilde f, \partial_t \tilde f \in \HS[1+\alpha]{\frac{1+\alpha}{4}}(\Sph^1 \times [0,T])\cap \CalC^\infty(\Sph^1 \times (0,T])$ as claimed and as above we have
$\partial_t \tilde f = 
-\nabla \E_\lambda(\tilde f)$. 
%
\end{proof}
\begin{remark}
 Due to the construction with the flow on the domain $\Sph^1 \times [0,T]$ in the proof of \autoref{thm:STE} we have no uniqueness for the solution $f \in \HS[1+\alpha]{\frac{1+\alpha}{4}}(\Sph^1 \times [0,T];\mathbb{H}^2)$ from above, but it follows from part 2) of the proof that any other solution $\tilde f$ that additionally satisfies $\partial_ t \tilde f, \partial_x \tilde f \in \HS[4+\alpha]{\frac{4+\alpha}{4}}$ equals a reparametrisation of the constructed solution $f$.
\end{remark}

In the next paragraphs we give a proof of the existence, uniqueness and smoothness of the quasilinear parabolic equation from \autoref{eq:reducedEq}, that is, the cited theorem \autoref{thm:nonlinear_schauder}. We postpone a few minor proofs to \autoref{sec:STE}. We start by giving an overview on parabolic H\"older spaces (c.f. \cite[Def 2.5.2]{Gerhardt}).

\subsection{Parabolic H\"older spaces and the linear problem}
\label{sec:hoelderspace}
Let $\alpha \in (0,1), k \geq 4$. The parabolic H\"older space of order $4$, $H^{k+\alpha, \frac{k+\alpha}{4}} (\Sph^1 \times [0,T])$, is the space of all functions $f \colon \Sph^1 \times [0,T]\to\R$ with continuous derivatives 
$\nabla^\beta \partial_t^\mu f$ for all $|\beta|+ 4\mu \leq k$ and finite norm
\begin{align*}
    \|f\|_{\HS[k+\alpha]{\frac{k+\alpha}{4}}(\Sph^1 \times [0,T])} \defeq 
 &
 \sum _{|\beta| + 4 \mu \leq k} \sup_{(x,t) \in \Sph^1 \times [0,T]} |\nabla^\beta \partial_t^\mu f| 
 +\sum _{|\beta| +4\mu = k} \sup_{t \in [0,T]}   [\nabla^\beta \partial_t^\mu f( \cdot, t)]_\alpha
    \\
  &+\!\!\!\sum _{0< k + \alpha - |\beta| - 4 \mu < 4} \sup_{x \in \Sph^1}
   [\nabla^\beta \partial_t^\mu f(x,\cdot)]_{\frac{1}{4}(k+\alpha -  |\beta| - 4 \mu)}.
  \end{align*}
 Here we write $\nabla$ for the covariant derivative on the Riemannian manifold $\Sph^1$, $\partial_t$ for the derivative with respect to $t\in[0,T]$ and define the H\"older seminorm for tensor fields $T \in \Gamma(T^{k,0}\Sph^1)$ as follows: $ [T]_\alpha = \sup_{x \neq y} \frac{|T(x) -\tau_{y,x}T(y)|_g}{d_g(x,y)^\alpha}$, where $\tau_{y,x}$ is the parallel transport from $y$ to $x$ and $d_g$ is the metric on $(\Sph^1,g)$. Similarly we denote the usual H\"older spaces on $\Sph^1$ by $\CalC^{k,\alpha}(\Sph^1)$.

Let us consider the following problem for 
$u \colon \Sph^1 \times [0,T]\to \R$.
\begin{equation}\label{eq:linearParaboliconM}
 \begin{cases}
  Lu \defeq  \partial_t u - \sum_{|\gamma|\leq 4} a_\gamma \nabla^\gamma u = f & \text{on }\Sph^1 \times [0,T)\\
  u(0) = u_0 & \text{on }\Sph^1.
 \end{cases}
\end{equation}
Under 
appropriate assumptions we find that $L$ is an isomorphism of Banach spaces. Here we use the notation $\CalC^{4+s} = \CalC^{\lfloor 4+s \rfloor,4+s - \lfloor 4+s \rfloor}$ for $s\notin \N_0$.
\begin{mythm} Let\label{thm:linearIsoSchauderM} $L$ be parabolic (in the sense of Petrovskii) and $a_\gamma \in 
\HS[s]{\frac{s}{4}}(\Sph^1 \times [0,T])$ for all $\gamma =0,1, \ldots, 4$ for some $s > 0$, $s \notin \N$. Then there exists some constant $C>0$ such that for all $f \in H^{s,\frac{s}{4}}(\Sph^1 \times [0,T])$ and $u_0 \in \CalC^{4+s}(\Sph^1 \times [0,T])$, there exists a solution $u \in \HS[4+s]{\frac{4+s}{4}} 
(\Sph^1 \times [0,T])$ to the problem \eqref{eq:linearParaboliconM}. The solution is unique and satisfies
\begin{equation}
\label{eq:norm_iso_L_schauderM}
\| u \| _ {\HS[4+s]{\frac{4+s}{4}}(\Sph^1 \times [0,T])
} \leq C \left(\|f\|_ {\HS[s]{\frac{s}{4}} (\Sph^1 \times [0,T])} + \| u_0 \| _ {\CalC^{4+s}(\Sph^1 \times [0,T])}\right).
\end{equation}
\end{mythm}

\autoref{thm:linearIsoSchauderM} follows from the classic Schauder theory for parabolic problems on domains. It  is given in the appendix in Paragraph \ref{proof:ProofLinearSchauderM}.

\subsection{The nonlinear problem - existence, uniqueness and smoothing}
Let $F\colon \Sph^1 \times \R^5 \times (0,\infty) \to \R$ be such that 
\[F(\cdot, u, \partial_x u, \ldots, \partial_x^4 u, |\partial_xf|_e^{-1}) =  - \frac{\partial_x^4 u}{|\partial_xf|^4_e}f_2^4 +P(\cdot,u,\partial_x u, \partial_x^2 u, \partial_x^3 u,|\partial_xf|_e^{-1})
\]
(c.f. \eqref{eq:reducedEq}). Let 
\begin{equation*}
 W = \{w \in \CalC^1(\Sph^1)\st f(x) = \overline f(x) + w(x)\overline  N(x) \text{ from \eqref{eq:ansatz_f} is an immersion of }\Sph^1\text{ into }\Hyp^2\}.
\end{equation*}
Then $W$ is an open subset containing $0$, and since $P(x,\ldots)$ is a polynomial with smooth coefficients we find that $F$ induces a smooth mapping
\[
  \mathbf{F}\colon \HS[4+\alpha]{\frac{4+\alpha}{4}} (\Sph^1 \times [0,T])\cap W \to \HS[\alpha]{\frac{\alpha}{4}} (\Sph^1 \times [0,T]), u \mapsto \mathbf{F}[u] \defeq F(\cdot, u, \partial_x u, \ldots, \partial_x^4 u, |\partial_xf|_e^{-1}),
\]
where we write $\HS[4+\alpha]{\frac{4+\alpha}{4}} (\Sph^1 \times [0,T])\cap W$ for the set of functions $u \in  \HS[4+\alpha]{\frac{4+\alpha}{4}} (\Sph^1 \times [0,T])$ that satisfy $u(\cdot,t) \in W$ for all $0 \leq t \leq T$.\\
For the derivative of $\mathbf{F}$ at $u \in \HS[4+\alpha]{\frac{4+\alpha}{4}}(\Sph^1 \times [0,T])\cap W$
we find
\begin{equation}\label{eq:DF[v]}
  D\mathbf{F}[u]v = \sum_{|\gamma|\leq 4} a_\gamma \partial_x^\gamma v
\end{equation}
with
\[a_4(x,t) = -\frac{f_2^4}{|\partial_xf|^4_e} = -\frac{ \overline f_2 + u(x,t)\overline N_2(x)}{|\partial_x(\overline f + u(x,t)\overline N(x))|^4_e} 
 \]
 and $
 a_\gamma(x,t) = \tilde a_\gamma(x,t, u,\partial_xu,\partial_x^2 u, \partial_x^3 u,|\partial_xf|^{-1}_e)$, 
 $\gamma = 0,\ldots,3,$ 
 for some smooth functions $\tilde a_\gamma\colon \Sph^1 \times [0,T] \times \R^4 \times (0,\infty) \to \R$ that are polynomials for fixed $(x,t)$. By continuity and compactness we find for $u(\cdot, 0) \in W$ some $\tilde \delta > 0, \delta > 0$ such that $f_2 \geq \tilde \delta$ and $\delta \leq -a_4(x,t) \leq \frac{1}{\delta}$ uniformly on $\Sph^1 \times [0,T]$, which shows that $L = \partial_t - D\mathbf{F}[{u}]$ is parabolic in the sense of Petrovskii.\\
 We can now solve \autoref{eq:reducedEq}, whose realisation in $\HS[4+\alpha]{\frac{4+\alpha}{4}} (\Sph^1 \times [0,\varepsilon])$ is now given by
\begin{equation} \label{eq:schauder_nonlinear}
 \left \{
\begin{array}{rl}
  \dot u &= \mathbf{F}[u]\\
  u(\cdot, 0) &= u_0.
\end{array}
\right.
\end{equation}

\begin{mythm}
 Let $u_0 \in \CalC^{4,\alpha}(\Sph^1)\cap W$.  \label{thm:nonlinear_schauder}
Then there exists some $\varepsilon > 0$ and a unique solution $u \in 
 \HS[4+\alpha]{\frac{4+\alpha}{4}} (\Sph^1 \times [0,\varepsilon])$ to the nonlinear problem \eqref{eq:schauder_nonlinear}.
\end{mythm}

\begin{proof}
We start with the existence part, the uniqueness is shown in the appendix in paragraph \ref{proof:uniqueness}. 
Since $\mathbf{F}[u_0] - D\mathbf F[0] u_0 \in\HS[\alpha]{\frac{\alpha}{4}} (\Sph^1 \times [0,T])$ by the remark above we find from \autoref{thm:linearIsoSchauderM} a unique solution $\tilde u \in  \HS[4+\alpha]{\frac{4+\alpha}{4}} (\Sph^1 \times [0,T])$ to the linear problem
 \begin{equation}\label{eq:linear_soln_tilde_u}
 \left \{
\begin{array}{rl}
  \dot {\tilde u} - D\mathbf{F}[0]\tilde u &=  \mathbf{F}[u_0] - D\mathbf{F}[0]u_0,\\
  \tilde u(\cdot, 0) &= u_0.
\end{array}
\right.
\end{equation}
If necessary, we can make $T$ smaller such that $\tilde u(\cdot,t) \in W$ for all $0 \leq t \leq T$.
Let $\tilde f \defeq \dot{\tilde u} - \mathbf{F}[\tilde u]$. Then $\tilde f \in \HS[\alpha]{\frac{\alpha}{4}}(\Sph^1 \times [0,T])$  and by smoothness of $\mathbf{F}$ we have
\[\tilde f(\cdot, 0) 
= D\mathbf{F}[0]\tilde u|_{t=0} + \mathbf{F}[u_0](\cdot,0) -D\mathbf{F}[0]u_0|_{t=0} - \mathbf{F}[u_0](\cdot,0) = 0.\]
To apply the Inverse Function Theorem we define for some $0 < \beta < \alpha$:
\begin{equation}
\Xcalnull_T \defeq \{\eta \in \HS[4+\beta]{\frac{4+\beta}{4}}(\Sph^1\times [0,T]) \st \eta(\cdot, 0) \equiv 0
\},\quad 
\label{eq:Y_0}
 \Ycalnull_T \defeq 
 \HS[\beta]{\frac{\beta}{4}}(\Sph^1\times [0,T])
\end{equation}
and
\begin{equation}
 \label{eq:defPhi}
 \Phi\colon \Xcalnull_T \to \Ycalnull_T, \qquad \eta \mapsto \dot{\tilde u} + \dot \eta -\mathbf{F}[\tilde u + \eta].
\end{equation}
It follows that $\Phi$ is well defined and smooth on a neighbourhood of $0$. 
 Moreover, $\Phi[0] = \tilde f$. Now we deduce from \autoref{thm:linearIsoSchauderM} that the linearisation
$D\Phi[0]{\eta} = \dot \eta - D\mathbf{F}[\tilde u ]\eta \defeq L_{\tilde u}\eta
$ 
is a linear isomorphism $L_{\tilde u}\colon \Xcalnull_T \to\Ycalnull_T$. Indeed the operator $L_{\tilde u}$ is the restriction of the parabolic operator of the problem
\begin{equation} \label{eq:schauder_new_linear}
 \left \{
\begin{array}{rl}
  L_{\tilde u}v &= g\\
v(0,\cdot) &= v_0
\end{array}
\right.
\end{equation}
on the space of vanishing trace at time $t=0$.  
Thus, the Inverse Function Theorem \cite[Theorem 4.F]{ZeidlerI} yields the existence of neighbourhoods $U \subset \Xcalnull_T$ of $0$, $V \subset \Ycalnull_T$ of $\Phi[0] = \tilde f$ such that
\begin{equation}\label{eq:PhiIso}
 \Phi\colon U \to V
\end{equation}
is a diffeomorphism. 
To show that $0 \in V$ we define for all $0 < \varepsilon < \min\{1,\frac{T}{2}\}$ a cut-off function $\phi_\varepsilon \in \mathcal{C}^\infty([0,T])$ satisfying
\[
 0 \leq \phi_\varepsilon \leq 1, \quad 0 \leq \dot \phi_\varepsilon \leq \frac{2}{\varepsilon}\quad  \text{ and }\quad 
\phi_\varepsilon(t) = \begin{cases}
                       0, & 0 \leq t \leq \varepsilon,\\
                       1, & 2\varepsilon \leq t \leq T.
                      \end{cases}
\]
Let $f_\varepsilon \defeq \phi_\varepsilon \tilde f \in \HS[\alpha]{\frac{\alpha}{4}}$. It follows from {\cite[Lemma 2.5.8]{Gerhardt}} that the family $(f_\varepsilon)_\varepsilon$ is uniformly bounded in $\HS[\alpha]{\frac{\alpha}{4}}(\Sph^1 \times [0,T])$.
%
%
The inclusion $\HS[\alpha]{\frac{\alpha}{4}}(\Sph^1 \times [0,T]) \hookrightarrow \HS[\beta]{\frac{\beta}{4}}(\Sph^1 \times [0,T])$ is compact by {Arzel\`{a}-Ascoli}, 
whence there exists some $\hat f \in \HS[\alpha]{\frac{\alpha}{4}}$ such that 
\[
 |f_{\varepsilon_k} - \hat f|_{\HS[\beta]{\frac{\beta}{4}}} \to 0
 \]
for some subsequence $\varepsilon_k \to 0$.
On the other hand,  $\tilde f$ satisfies $\tilde f(\cdot, 0) = 0$ and is continuous, hence $f_\varepsilon \to \tilde f$ uniformly 
as $\varepsilon \to 0$, thus we find $\tilde f = \hat f$ and
$ f_{\varepsilon_k} \to \tilde f\text{ in }\Ycalnull_T.
$
Thus there exists some $\varepsilon \defeq \varepsilon_{k_0} > 0$ such that $f_\varepsilon \in V$. Since $\Phi\colon U \to V$ is a diffeomorphism there exists some $\eta \in U$ such that$  f_\varepsilon = \Phi[\eta]= \dot{\tilde u} + \dot \eta -F[\tilde u + \eta]$. This implies that $u \defeq \tilde u + \eta \in \HS[4+\beta]{\frac{4+\beta}{4}}(\Sph^1 \times [0,T])$ satisfies
\begin{equation*}
 \left \{
\begin{array}{rl}
  \dot u &= \mathbf{F}[u] + f_\varepsilon\\
  u(\cdot, 0) &= 0+ \tilde u(\cdot, 0) = u_0.
\end{array}
\right.
\end{equation*}
Applying the definition of $f_\varepsilon$ we find that in particular $u$ solves \eqref{eq:schauder_nonlinear} on $\Sph^1 \times [0,\varepsilon]$.\\
To show that $u \in \HS[4+\alpha]{\frac{4+\alpha}{4}}(\Sph^1 \times [0,\varepsilon])$ we let
 $\tilde a_4(x,t) \defeq \displaystyle \frac{f_2^4}{|\partial_xf|_e^4}
 $ (where we set $f = \overline f + u\overline N$ as in \eqref{eq:ansatz_f}), then $\tilde a_4 \in \HS[3+\beta]{\frac{3+ \beta}{4}}\subset \HS[\alpha]{\frac{\alpha}{4}}$ 
and also $\tilde f(x,t) \defeq P(x,t,u,\ldots,\partial_x^3 u, |\partial_xf|_e^{-1})\in \HS[1+\beta]{\frac{1+ \beta}{4}}\subset \HS[\alpha]{\frac{\alpha}{4}}$. 
 Thus $u \in \HS[4+\beta]{\frac{4+\beta}{4}} (\Sph^1 \times [0,\varepsilon])$ is the unique solution to the linear problem 
 \begin{equation} \label{eq:schauder_nonlinear_bootstrapping}
 \left \{
\begin{array}{rl}
  \partial_t u - \tilde a_4 (x,t) \partial_x^4 u &= \tilde f(x,t)\\
 u(0,\cdot) &= u_0
\end{array}
\right.
\end{equation}
with data $u_0 \in \CalC^{4,\alpha}$, $\tilde f \in \HS[\alpha]{\frac{\alpha}{4}}$ and coefficients in $\HS[\alpha]{\frac{\alpha}{4}}$.
\autoref{thm:linearIsoSchauderM} yields the existence of a unique solution $\tilde u \in \HS[4+\alpha]{\frac{4+\alpha}{4}} (\Sph^1 \times [0,\varepsilon])$. Thus $u = \tilde u$ and $u$ has the desired smoothness.
\end{proof}

One can apply the same bootstrapping argument as at the end of the proof of \autoref{thm:nonlinear_schauder} to show that the solution is smoother if $u_0$ has more regularity.
\begin{corol}
\label{corol:smoothness_soln}
  \label{corol:more_regularity_of_soln}If the initial value from \autoref{thm:nonlinear_schauder} satisfies $u_0 \in \CalC^{4+m,\alpha}(\Sph^1)$, then the solution $u$ also satisfies $u \in \HS[4+m+\alpha]{\frac{4+m+\alpha}{4}}(\Sph^1 \times [0,\varepsilon])$. In particular we see that $u \in \CalC^\infty(\Sph^1 \times [0,\varepsilon])$ if  $u_0 \in \CalC^\infty(\Sph^1)$.
\end{corol}




If we can not apply \autoref{corol:smoothness_soln} we still have the following parabolic smoothing:
\begin{mythm}
 The \label{thm:parabolic_smoothing}solution $u \in \HS[4+\alpha]{\frac{4+\alpha}{4}}(\Sph^1\times [0,\varepsilon])$ of \autoref{eq:schauder_nonlinear} from \autoref{thm:nonlinear_schauder} satisfies
  \[
    u \in \CalC^\infty(\Sph^1 \times [\delta, \varepsilon]) \text{   for all } 0 < \delta < \varepsilon.
  \]
\end{mythm}

 The proof is given in paragraph \ref{proof:ProofOfSmoothing} in the appendix.

\section{Long time existence}
\label{sec:LTE}
By the main result in the previous section the solution to \eqref{eq:stat} exists on at least a small interval of time and we extend it to its maximal existence interval $[0,T_{\text{max}})$. In order to prove our main result Theorem \ref{mainThm} we show first that $T_{\text{max}}= \infty$ by proving that, if this was not the case, the derivatives of the curvature are uniformly bounded on $[0,T_{\text{max}})$ using interpolation inequalities.

\subsection{Interpolation inequalities}

As we have seen in Lemma \ref{lem:dercurvature} the evolution equations of the derivatives of the curvature are quite complicated and with several terms. With the notation $P^{a,c}_b(\Kapp)$ we keep track of the order of these terms that we wish now to control via interpolation inequalities.  

The norms we use are the following. For a function $h: (\mathbb{S}^1, \diff s) \to \R$ and $k \in \N$, $p \in [1,\infty)$
$$ \| h\|^p_{L^p(\mathbb{S}^1)} = \int_{\mathbb{S}^1} |h(s)|^p \diff s , \; \| h\|_{L^\infty(\mathbb{S}^1)} = \esssup |h(x)| \mbox{ and } 
\| h\|^2_{W^{k,2}(\mathbb{S}^1)} = \sum_{j=0}^k \| \partial_s^j h\|^2_{L^2(\mathbb{S}^1)} \, , $$
while for a vector field $\Phi: (\mathbb{S}^1, \diff s) \to T \Hyp^2$
\begin{align*}
& \| \Phi\|^p_{L^p(\mathbb{S}^1)} = \int_{\mathbb{S}^1} |\Phi(s)|_g^p \diff s = \| |\Phi|_g \|^p_{L^p(\mathbb{S}^1)}, \; \| \Phi \|_{L^\infty(\mathbb{S}^1)} = \esssup |\Phi(x)|_g \\
&  \mbox{ and } 
\| \Phi \|^2_{W^{k,2}(\mathbb{S}^1)} = \sum_{j=0}^k \| (\nabla_{\partial s}^\bot)^j \Phi\|^2_{L^2(\mathbb{S}^1)} \, .
\end{align*}
In the case $M=\R^n$ it is convenient to work with scale invariant norms. There is no need to modify the norms here since the metric in $\Hyp^2$ is already scaling invariant. 

First an interpolation inequality for the derivatives of the curvature.
\begin{prop}\label{prop:propsonntag}
Let $f:\mathbb{S}^1 \to \Hyp^2 $ be a smooth immersion such that $\int_{\mathbb{S}^1} \diff s=L>0$ with $\diff s = |\partial_x f|_g \diff x$. Let $\Kapp$ be the curvature of $f$. Then for any $k \in \N$, $0\leq i<k$ and $p \in [2,\infty]$ there exists a constant $c$ depending only on $i$, $k$, $p$ and $1/L$  such that
$$ \| (\nabla_{\partial_s}^\bot)^i \Kapp \|_{L^p(\mathbb{S}^1)} \leq c \| \Kapp  \|_{W^{k,2}(\mathbb{S}^1)}^{\alpha}  \|\Kapp\|_{L^2(\mathbb{S}^1)}^{1-\alpha} \,  , $$
with $\alpha=(i+1/2-1/p)/k$ (and $\alpha=(i+1/2)/k$ if $p=\infty$).
\end{prop}
The proof is given in Appendix \ref{sec:Inter}. A consequence of this result is that the $W^{k,2}$-norm of the curvature is bounded by the $L^2$-norm of the curvature and by the $L^2$-norm of the highest derivative.

\begin{corol}\label{cor:corsonntag}
Consider the same assumptions of Proposition \ref{prop:propsonntag}. Then for any $k \in \N$ there exists a constant $c$ depending only on $k$ and $1/L$  such that
$$ \| \Kapp \|_{W^{k,2}(\mathbb{S}^1)} \leq c ( \| (\nabla_{\partial_s}^\bot)^k\Kapp  \|_{L^{2}(\mathbb{S}^1)}+ \|\Kapp\|_{L^2(\mathbb{S}^1)}) \,  . $$
\end{corol}
\begin{proof}
The estimate for $k=1$ is satisfied with $c=1$ by definition of the norm. The general case is then proven by induction using Proposition \ref{prop:propsonntag}. The details are given in \cite[Cor.4.2]{AnnaPaola1}.
\end{proof}

We are now ready to state the interpolation inequality in the form needed in the proof of the long time existence. More precisely, we see which estimate we can get for the terms $P^{a,c}_b(\Kapp)$ (defined just before Lemma \ref{lem:dercurvature}) in terms of $\|\Kapp  \|_{L^{2}(\mathbb{S}^1)}$ and $\|(\nabla_{\partial_s}^\bot)^k\Kapp  \|_{L^{2}(\mathbb{S}^1)}$ for some $k$. Here with abuse of notation we write $|P^{a,c}_b(\Kapp)|_g$ both when $b$ is even and odd. For $b$ even, $|P^{a,c}_b(\Kapp)|_g=|P^{a,c}_b(\Kapp)|$. 

\begin{prop}\label{prop:inter}
Let $f:\mathbb{S}^1 \to \Hyp^2 $ be a smooth immersion such that $\int_{\mathbb{S}^1} \diff s=L>0$ with $\diff s = |\partial_x f|_g \diff x$. Then for any $k \in\N$, $a,c \in\N_0$, $b \in \N$, $b\geq 2$ such that $ c\leq k-1$ and $a+\frac{b}{2} -1<2k,$ 
there exists a constant $C$ depending on $k$, $a$, $b$ and $\frac{1}{L}$ such that
$$\int_{\mathbb{S}^1} |P^{a,c}_b(\Kapp)|_g\diff s \leq C \|\Kapp  \|_{W^{k,2}(\mathbb{S}^1)}^{\gamma}\|\Kapp\|^{b-\gamma}_{L^2(\mathbb{S}^1)}\, ,$$
with $\gamma=(a+b/2-1)/k$. Moreover, for any $\varepsilon \in(0,1)$
$$\int_{\mathbb{S}^1} |P^{a,c}_b(\Kapp)|_g\diff s \leq \varepsilon \| (\nabla_{\partial_s}^\bot)^k\Kapp  \|_{L^{2}(\mathbb{S}^1)}^2 + \tilde{c} \varepsilon^{-\frac{\gamma}{2-\gamma}} \|\Kapp  \|_{L^{2}(\mathbb{S}^1)}^{2\frac{b-\gamma}{2-\gamma}}+\tilde{c} \|\Kapp\|^{b}_{L^2(\mathbb{S}^1)}\,,$$
with $\tilde{c}$ depending on $k$, $a$, $b$ and $\frac{1}{L}$.
\end{prop}
\begin{proof}
We start by proving the first inequality. If $\gamma=0$, i.e. $a=0$ and $b=2$ then
$$\int_{\mathbb{S}^1} |P^{0,c}_2(\Kapp)|_g\diff s = \|\Kapp\|^{2}_{L^2(\mathbb{S}^1)} ,$$
and the estimate is then true taking any $C \geq 1$. In the general case by the Cauchy-Schwarz inequality (both for $b$ even and odd)
$$ |P^{a,c}_b(\Kapp)|_g \leq |(\nabla_{\partial_s}^\bot)^{i_1}\Kapp|_g |(\nabla_{\partial_s}^\bot)^{i_2}(\Kapp)|_g \dots |(\nabla_{\partial_s}^\bot)^{i_b}(\Kapp)|_g \,,$$
with $i_j\leq c$ and $\sum_{j=1}^b i_j=a$. So with H\"older's inequality and \autoref{prop:propsonntag}
\[
\int_{\mathbb{S}^1} |P^{a,c}_b(\Kapp)|_g\diff s 
\leq  \prod_{j=1}^b \| (\nabla_{\partial_s}^\bot)^{i_j}\Kapp  \|_{L^{b}(\mathbb{S}^1)} 
\leq  \prod_{j=1}^b  c_j \| \Kapp  \|_{W^{k,2}(\mathbb{S}^1)}^{\alpha_j}  \|\Kapp\|_{L^2(\mathbb{S}^1)}^{1-\alpha_j}
\]
with $\alpha_j= (i_j+1/2-1/b)/k$. Since $\sum_{j=1}^b \alpha_j=\gamma$ the first estimate follows directly.

For the second estimate we bound $\| \Kapp  \|_{W^{k,2}(\mathbb{S}^1)}$ using Corollary \ref{cor:corsonntag} obtaining
$$\int_{\mathbb{S}^1} |P^{a,c}_b(\Kapp)|_g\diff s \leq \tilde{c}_1 \|(\nabla_{\partial_s}^\bot)^{k} \Kapp \|_{L^{2}(\mathbb{S}^1)}^{\gamma} \|\Kapp\|^{b-\gamma}_{L^2(\mathbb{S}^1)}+ \tilde{c}_1 \|\Kapp\|^{b}_{L^2(\mathbb{S}^1)}\, ,$$
and then use that 
$$ab \leq\varepsilon a^p + \frac{1}{(\varepsilon p)^{q/p}}b^q \frac{1}{q} \mbox{ with }p=\frac{2}{\gamma},\ q=\frac{2}{2-\gamma}  \mbox{ and } a=\|(\nabla_{\partial_s}^\bot)^{k} \Kapp \|_{L^{2}(\mathbb{S}^1)}^{\gamma}\,.$$\qedhere
\end{proof}

\subsection{Proof of Theorem \ref{mainThm}}

\begin{proof}[Proof of Theorem \ref{mainThm}] \textbf{Short time Existence} By Theorem \ref{thm:STE} we know that if the initial datum satisfies $f_0 \in C^{5,\alpha}$ then there exists a solution to \eqref{eq:stat} in the H\"older space $\HS[1+\alpha]{\frac{1+\alpha}{4}}(\Sph^1\times[0,T])$.  Moreover, the solution is an immersion and is smooth on $[\delta,T]$ for any $\delta>0$. Let us fix $\delta= \frac12 T$ and the constants 
\begin{equation}\label{eq:Am}
A_m:= \sum_{i=0}^m \int_{\mathbb{S}^1} | (\nabla_{\partial_s}^\bot)^i\Kapp(\cdot, \delta)|^2_g \diff s < \infty \, .
\end{equation}

\textbf{Global Existence} Let $[0,T_{\text{max}})$ be the maximal existence interval for the solution and let us assume that $T_{\text{max}}< \infty$. Being \eqref{eq:stat} an $L^2$-gradient flow for the energy $\E_{\lambda}$, the $L^2$-norm of the curvature is already bounded by $\E_{\lambda}(f_0)$. Hence there exists a constant $C=C(\E_{\lambda}(f_0))$ such that
$$ \| \Kapp\|_{L^2(\mathbb{S}^1)} \leq C \mbox{ for all }t \in [0,T_{\text{max}}) . $$
We prove now that the solution satisfies uniform bounds on $[0,T_{\text{max}})$ and hence can be extended, reaching a contradiction as done in \cite{DKS}. In the following $C$ is a constant that might change from line to line. We will at each step specify on which parameters the constant depends.\medskip

\textit{Step 1.} \underline{Along the flow the length is uniformly bounded from below.} For $t \in [\delta,T_{\text{max}})$ let $\mathcal{L}(f(t))$ denote the total length of $f(t): \Sph^1 \to \Hyp^2$ solution of \eqref{eq:stat}. Then by Fenchel's Theorem in the hyperbolic plane (Theorem \ref{Thm:lengthbelow}) we find for any $t \in [\delta,T_{\text{max}})$ 
$$ 2 \pi \leq \int_{0}^{\mathcal{L}(f(t))} | \Kapp(t,s)|  \diff s \leq \sqrt{\mathcal{L}(f(t))} \sqrt{\mathcal{E}(f(t))} \leq  \sqrt{\mathcal{L}(f(t))} \sqrt{ \mathcal{E}_{\lambda}(f(0))} \, ,$$
that gives a uniform bound from below on the length independent of $t$.\medskip

\textit{Step 2.} \underline{Uniform bounds on $\|(\nabla_{\partial_s}^\bot)^m\Kapp\|_{L^2}$.} By Lemma \ref{lem:dercurvature} (with $S_0=-1$), Lemma \ref{lem:integration} with $N = (\nabla_{\partial_s}^\bot)^m\Kapp$ and
\begin{equation}\label{VV}
V=  -(\nabla_{\partial_s}^\bot )^2 \Kapp+ P_3^{0,0}(\Kapp) +P_1^{0,0}(\Kapp) \, ,
\end{equation} 
(see \eqref{eq:GradAsPolynom}) we find for any $m\geq 1$
\allowdisplaybreaks{\begin{align}\nonumber
  & \frac{\diff}{\diff t} \frac{1}{2} \int_{\mathbb{S}^1} | (\nabla_{\partial_s}^\bot)^m\Kapp|^2_g \diff s + \int_{\Sph^1} | (\nabla_{\partial_s}^\bot)^{m+2}\Kapp |_g^2 \diff s 
  \nonumber \\  
	& = \int_{\mathbb{S}^1}  (P_4^{2+2m, 2+m}(\Kapp) + P_2^{2+2m,2+m}(\Kapp) 
	\label{eq:sole1}
	+ P_6^{2m,m}(\Kapp) +P_4^{2m,m}(\Kapp) + P_2^{2m,m}(\Kapp)) \diff s .
 \end{align}}
We estimate now the terms on the right hand side using interpolation inequalities. Since $\mathbb{S}^1$ has no boundary and by Proposition \ref{prop:inter} we find for any $\varepsilon_1 \in (0,1)$ 
$$  \int_{\mathbb{S}^1}  P_4^{2+2m, 2+m}(\Kapp)\diff s =  \int_{\mathbb{S}^1}  P_4^{2+2m, 1+m}(\Kapp)\diff s \leq \varepsilon_1  \int_{\Sph^1} | (\nabla_{\partial_s}^\bot)^{m+2}\Kapp |_g^2 \diff s + C_1 $$
with $C_1= C_1(\varepsilon_1, \E_{\lambda}(f_0), m)$ since $ \frac{2+2m+2-1}{m+2} <2$, and the length is uniformly bounded from below. Similarly 
\begin{align*}
\int_{\mathbb{S}^1} P_2^{2+2m,2+m}(\Kapp)\diff s = \int_{\mathbb{S}^1} P_2^{2+2m,1+m}(\Kapp)\diff s \leq \varepsilon_2  \int_{\Sph^1} | (\nabla_{\partial_s}^\bot)^{m+2}\Kapp |_g^2 \diff s + C_2 \, , 
\end{align*}
with $C_2= C_2(\varepsilon_2,\E_{\lambda}(f_0), m)$. For the other terms (now there is no need of integrating by parts)
\begin{align*}
\int_{\mathbb{S}^1} P_6^{2m,m}(\Kapp) \diff s  & \leq \varepsilon_3  \int_{\Sph^1} | (\nabla_{\partial_s}^\bot)^{m+2}\Kapp |_g^2 \diff s + C_3, \\
\int_{\mathbb{S}^1} (1+\lambda) P_4^{2m,m}(\Kapp)\diff s  & \leq \varepsilon_4  \int_{\Sph^1} | (\nabla_{\partial_s}^\bot)^{m+2}\Kapp |_g^2 \diff s + C_4 \mbox{ and }\\
\int_{\mathbb{S}^1} (1+\lambda) P_2^{2m,m}(\Kapp)\diff s  & \leq \varepsilon_5  \int_{\Sph^1} | (\nabla_{\partial_s}^\bot)^{m+2}\Kapp |_g^2 \diff s + C_5 \, ,
\end{align*}
with $C_{i}= C_{i}(\varepsilon_{i}, \lambda,\E_{\lambda}(f_0), m)$, $i=3,4,5$. Combining these inequalities and choosing $\varepsilon_1=...= \varepsilon_5=1/10$ we find from \eqref{eq:sole1}
\begin{align}\label{eq:sole2}
  & \frac{\diff}{\diff t} \frac{1}{2} \int_{\mathbb{S}^1} | (\nabla_{\partial_s}^\bot)^m\Kapp|^2_g \diff s + \int_{\Sph^1} | (\nabla_{\partial_s}^\bot)^{m+2}\Kapp |_g^2 \diff s \leq \frac12 \int_{\Sph^1} | (\nabla_{\partial_s}^\bot)^{m+2}\Kapp |_g^2 \diff s + C \, ,
	\end{align}
	with $C= C(\lambda, \E_{\lambda}(f_0), m)$. Summing on both sides of the inequality above the term $\frac{1}{2} \int_{\mathbb{S}^1} | (\nabla_{\partial_s}^\bot)^m\Kapp|^2_g \diff s$,\"o and since by Proposition \ref{prop:inter} 
	$$ \frac{1}{2} \int_{\mathbb{S}^1} | (\nabla_{\partial_s}^\bot)^m\Kapp|^2_g \diff s \leq \frac12 \int_{\Sph^1} | (\nabla_{\partial_s}^\bot)^{m+2}\Kapp |_g^2 \diff s + C , $$
	with $C= C(\E_{\lambda}(f_0), m)$, it follows from \eqref{eq:sole2} that
	\begin{align*}
  & \frac{\diff}{\diff t} \frac{1}{2} \int_{\mathbb{S}^1} | (\nabla_{\partial_s}^\bot)^m\Kapp|^2_g \diff s + \frac{1}{2} \int_{\mathbb{S}^1} | (\nabla_{\partial_s}^\bot)^m\Kapp|^2_g \diff s \leq  C \,,
	\end{align*}
	with $C= C(\lambda, \E_{\lambda}(f_0), m)$. The above differential inequality together with \eqref{eq:Am} imply that for any $m \in \N$
	\begin{equation}\label{eq:estKmnorm}
	\int_{\mathbb{S}^1} | (\nabla_{\partial_s}^\bot)^m\Kapp|^2_g \diff s \leq C = C(A_m,\lambda, \E_{\lambda}(f_0), m) < \infty  \mbox{ for all }t \in [\delta, T_{\text{max}}) \, .
	\end{equation}
	Notice that the constant is independent of $t$.\medskip
	
	\textit{Step 3.} \underline{Uniform bounds on $\||\nabla_{\partial_s}^m\Kapp|_{g}\|_{L^{\infty}}$.} Now we control not only the normal component of the derivative but the entire derivative. By Lemma \ref{lem:dercurvnormnicht} it follows that
	\begin{align*}
   \int_{\mathbb{S}^1} | \nabla_{\partial_s}^m\Kapp|^2_g \diff s & \leq 
	C(m) \int_{\mathbb{S}^1} |(\nabla_{\partial_s }^\bot)^m\Kapp |^2_g \diff s
  + C(m) \sum_{b=2,\; b\text{ even}}^{m+1} \int_{\mathbb{S}^1} P_{2b}^{2m+2 -2b, m+1-b}(\Kapp)  \diff s \\
   & +  C(m) \sum_{b=3,\; b\text{ odd}}^{m+1} \int_{\mathbb{S}^1} P_{2b}^{2m+2 -2b, m+1-b}(\Kapp) \diff s.
	\end{align*}
	since $\langle \sum_{i=1}^n a_i, \sum_{i=1}^n a_i \rangle_g \leq 2^{n-1} \sum_{i=1}^n |a_i|_{g}^2$. Using again the interpolation inequality given in Proposition \ref{prop:inter} (with $k=m$) we find we find for each term in the sums for $b$ even or odd
	\begin{align*}
	\int_{\mathbb{S}^1} P_{2b}^{2m+2 -2b, m+1-b}(\Kapp)  \diff s & \leq \int_{\mathbb{S}^1} | (\nabla_{\partial_s}^\bot)^m\Kapp|^2_g  + C, 
	\end{align*}
	with $C= C(\lambda, \E_{\lambda}(f_0), m)$. Combining these estimates with \eqref{eq:estKmnorm} we obtain
	\begin{equation}\label{eq:estKm}
	\int_{\mathbb{S}^1} | \nabla_{\partial_s}^m\Kapp|^2_g \diff s \leq C = C(A_m,\lambda, \E_{\lambda}(f_0), m) < \infty  \mbox{ for all }t \in [\delta, T_{\text{max}}) \, .
	\end{equation}
	Since the length of the curves is uniformly bounded from below from Proposition \ref{prop:propsonntag} and Lemma \ref{lem:dercurvnormnicht} it follows that for any $m \in \N$
	\begin{equation}\label{eq:estKminfty}
	\| |\nabla_{\partial_s}^m\Kapp|_g \|_{\infty} \leq C = C(A_m,\lambda, \E_{\lambda}(f_0), m) < \infty  \mbox{ for all }t \in [\delta, T_{\text{max}}) \, .
	\end{equation}

	\textit{Step 4.} \underline{Uniform bounds on $\||\nabla_{\partial_x}^m\Kapp|_{g}\|_{L^{\infty}}$.} In order to extend the solution to time $t=T_{max}$ control of the norms of $\nabla_{\partial_x}^m\Kapp$ is needed. For this and according to Lemma \ref{lem:controlparsuff} we need a control of the parametrisation, that is of $\gamma:=|\partial_x f|_g$ and its derivatives. We start by deriving estimates from above and below for $\gamma$.
	
	By \eqref{eq:evolutionoflineelement1} we see that $\partial_t \gamma = - \langle \Kapp, V\rangle_g \gamma$ with $V$ as in \eqref{VV}. Due to \eqref{eq:estKminfty} the coefficient $\langle \Kapp, V\rangle_g$ is bounded  in $L^{\infty}$ for all $t \in [\delta, T_{\text{max}})$. Since the solution at time $t=\delta$ is an immersion, there exists a $\mu>0$ such that $0 < \mu  \leq |\partial_x f(\delta)|_g \leq \frac{1}{\mu} < \infty$. Combining this two facts one finds the existence of a constant $C=C(T_{\text{max}})$ such that
	$$0 < \frac{1}{C}  \leq \gamma=|\partial_x f(t)|_g \leq C < \infty \mbox{ for all }t \in [\delta, T_{\text{max}}) \mbox{ on  }\mathbb{S}^1\, .$$
Since the derivatives of $\gamma$ satisfy the ordinary differential equation
$$ \partial_t \partial_x^m \gamma = - \langle \Kapp, V\rangle_g \partial_x^{m}\gamma + \sum_{j=0}^{m-1} c_{m,j} (\partial_x^{m-j} \langle \Kapp, V\rangle_g) \partial_x^{j}\gamma \, , $$
 with constants $c_{m,j}$, one proves with the same arguments and by induction that there exist constants $C=C(m,T_{\text{max}})$ such that
$$| \partial_x^{m} \gamma| \leq C \mbox{ for all }t \in [\delta, T_{\text{max}}) \mbox{ on  }\mathbb{S}^1\, .$$
Hence, from Lemma \ref{lem:controlparsuff} and \eqref{eq:estKminfty} it follows that for any $m \in \N$
	\begin{equation}\label{eq:estKminftydx}
	\| |\nabla_{\partial_x}^m\Kapp|_g \|_{\infty} \leq C = C(A_m,\lambda, \E_{\lambda}(f_0), m, T_{\text{max}}) < \infty  \mbox{ for all }t \in [\delta, T_{\text{max}}) \, .
	\end{equation}
	
	 \underline{\textit{Conclusion.}} Since by the estimates above in finite time the length remains bounded, $(f(t))_{t \in [\delta, T_{\text{max}})}$ remains in a compact subset of $\Hyp^2$. Having  uniform estimates on $f$ and all its derivatives on $[\delta,  T_{\text{max}}) \times \mathbb{S}^1$, we can extend the solution up to time $t=T_{\text{max}}$. Then at time $T_{\text{max}}$ we have a $C^{\infty}$-initial datum $f(T_\infty)$ and we can restart the flow, obtaining a smooth solution in  $[\delta,  T_{\text{max}}+\varepsilon)$ for some $\varepsilon>0$ (by Theorem \ref{thm:STE}), which contradicts our assumption. 
	Hence $T_{\text{max}}=\infty$.\medskip

\textbf{Subconvergence for $\lambda>0$}: Let $(t_k)_{k \in \N}$ be a sequence of times diverging to $+ \infty$ and  $(f(t_k, \cdot))_{k \in \N}$ be parametrised by constant speed with parameter in $[0,1]$, i.e. $|\partial_x f(t_k, x)|_g= L(f(t_k))/(2 \pi)$, $x \in \mathbb{S}^1$. Since $\lambda >0$ and the flow reduces the energy, we see that the length of the curves $(f(t_k, \cdot))_{k \in \N}$ is uniformly bounded. Let $L_0$ be the supremum of those lengths. By the estimates  obtained in the first part of the proof $\||  \nabla_{s_{f}}^m \Kapp_{f}(t_k)|_g\|_{L^{\infty}}$ are uniformly bounded  (\eqref{eq:estKminfty}).

Let now take vectors $((p_k,0)^t)_{k \in \N} \in\R^2$ such that
$$ \{ f(t_k, x)- (p_k,0)^t: \, x\in \mathbb{S}^1\} \cap \{ (0, y)^t: \, y>0\}\ne \emptyset\,, $$
and $(\alpha_k)_{k \in\N}$ be a sequence of positive numbers such that the rescaled curves 
$$ \hat{f}(t_k,\cdot) = \alpha_k( f(t_k, \cdot)- (p_k,0)^t) $$
go through the point $Q=(0,2 L_0)^t$. By Remark \ref{rem:scaling} and \eqref{eq:estKminfty} we see that $ |\partial_x \hat{f}(t_k)|_g $ and $ \||  \nabla_{s_{f}}^m \Kapp_{\hat{f}}(t_k)|_g\|_{L^{\infty}}$
are uniformly bounded. By construction and Remark \ref{rem:stern} we have also achieved that there exist $\delta>0$ and $M>0$ such that
$$ \hat{f}_2(t_k,x) \geq \delta>0 \mbox{ and } \| \hat{f}(t_k,x)\|_{\text{euc}} \leq M<\infty \quad \forall k \in\N, \forall x \in \mathbb{S}^1\,. $$
Hence $(\hat{f}(t_k))_{k \in \N}$ are uniformly bounded in $W^{m,2}(\mathbb{S}^1,g)$, for all $m \in \N$, and the weight $g$ is uniformly bounded since the sequence stays in a compact subset of $\Hyp^2$. It follows that there exists a subsequence $(t_{k_j})_{j \in\N}$ and $\hat{f}$ smooth such that $\hat{f}(t_{k_j}) \to \hat{f}$ in any $W^{m,2}(\mathbb{S}^1,g)$.

We prove now that the limit is a critical point of the elastic energy with the usual argument. Let $u(t)=\| |V|_g\|_{L^2}^2(t)$ with $V= - \nabla_{L^2} \E_{\lambda}$ as in \eqref{VV}. By \eqref{eq:endecrease} $u$ is integrable on $[0,\infty)$. In order to derive that it has zero limit for $t \to \infty$ we show that it is not oscillating. Indeed by \eqref{eq:evolutionoflineelement2}
\[\frac{d}{dt} u(t)= -\int_{\Sph^1} |V|^2\langle \Kapp, V\rangle_g \; \diff s_g +  \int_{\Sph^1}\langle V,\nabla_{\partial_t}^{\perp} V \rangle_g \; \diff s_g \]
with 
\begin{align}\label{eq:sabbath}
\nabla_{\partial_t}^{\perp} V &= \sum_{i=0}^3 \sum_{\substack{j=1\\j \text{ odd}}}^{7-2i} P^{2i,2i}_{j}(\Kapp) ,
\end{align}
see the proof in the appendix page \pageref{sabbath}.
By the uniform bounds in \eqref{eq:estKminfty} it follows that $|\frac{d}{dt} u(t)| \leq C$ and hence that $u(t) \to 0$ for $t \to\infty$. Therefore $\hat{f}$ is a critical point of the elastic energy.
	\end{proof}

\appendix
\section{Technical proofs}\label{sec:tech}

\begin{proof}[{Proof of Lemma \ref{lemma:evo}}]\label{proof:lemma_evo}
By \eqref{eq:CovDerLocally}, since $\Gamma_{ij}^k=\Gamma_{ji}^k$ and since $f$ is smooth we obtain
 \begin{align*}
 \nabla_{\partial_t}\partial_x f &= \sum_{k=1}^n (\partial_t \partial_x f_k + \sum_{i,j=1}^{n} \partial_x f_i \partial_t f_j \Gamma_{ij}^k)\partial_{y_k}\\
 &= \sum_{k=1}^n (\partial_x \partial_t f_k + \sum_{i,j=1}^n \partial_x f_i \partial_t f_j \Gamma_{ji}^k)\partial_{y_k} =\nabla_{\partial_x}\partial_t f ,
 \end{align*}
 that is \eqref{eq:commutingcoordinates}.  By the compatibility of the metric, the evolution $\partial_t f = V +\phi \partial_s f$ and \eqref{eq:commutingcoordinates}
 \begin{align*}
 \partial_t |\partial_x f|_g & 
 = \frac{1}{|\partial_x f|_g} \langle \nabla_{\partial_t }\partial_xf , \partial_x f\rangle_g
 = \frac{1}{|\partial_x f|_g} \langle \nabla_{\partial_x }\partial_tf , \partial_x f\rangle_g 
 \\
   &= \partial_x \langle \partial_tf , \partial_s f\rangle_g -\langle \partial_tf , \nabla_{\partial_x}\partial_s f\rangle_g  =(\partial_s \phi - \langle V,\Kapp\rangle_g )|\partial_xf|_g \,,
 \end{align*}
since $\Kapp=\nabla_{\partial_s}\partial_s f$ and $\partial_x =|\partial_xf|_g  \partial_s$. This gives \eqref{eq:evolutionoflineelement1}. Formula \eqref{eq:evolutionoflineelement2} is a direct consequence of \eqref{eq:evolutionoflineelement1} and $\diff s = |\partial_x f|_g \diff x$.

Formula \eqref{eq:commutator1} follows from the product rule, \eqref{eq:commutingcoordinates} and \eqref{eq:evolutionoflineelement1}. Indeed,
\begin{align*}
 \nabla_{\partial_t}\partial_s f - \nabla_{\partial_s}\partial_t f 
 &= \big( \partial_t \frac{1}{|\partial_xf|_g} \big) \partial_x f +\frac{1}{|\partial_xf|_g} \nabla_{\partial_t}{\partial_x f} - \frac{1}{|\partial_xf|_g}\nabla_{\partial_x}\partial_t f \\
 &= -\frac{1}{|\partial_xf|^2_g}(\partial_s \phi - \langle V,\Kapp\rangle_g )|\partial_xf|_g{\partial_x f} = (\langle V,\Kapp\rangle_g -\partial_s \phi)\partial_s f.
 \end{align*}
Since $N$ is a vector field normal to $f$ and by definition of $\nabla_{\partial_s}^\bot$ we find
\begin{align*}
\nabla_{\partial_s} N - \nabla_{\partial_s}^\bot N 
&= \langle{\nabla_{\partial_s}N, \partial_sf }\rangle_g \partial_s f 
= (\partial_s \langle N, \partial_s f\rangle _g - \langle N, \nabla_{\partial_s} \partial_sf \rangle_g)\partial_sf=- \langle N, \Kapp\rangle _g \partial_s f.
\end{align*}
that is \eqref{eq:NormalDerOfNormal}. Formula \eqref{eq:EvoUnitVelocity} for the evolution of the tangent vector is a consequence of \eqref{eq:commutator1} and \eqref{eq:NormalDerOfNormal} since
\begin{align*}
 \nabla_{\partial_t} \partial_s f 
 &= \nabla_{\partial_s}( V + \phi \partial_s f)
 + (\langle V,\Kapp\rangle_g -\partial_s \phi)\partial_s f\\
   &= \nabla_{\partial_s}^\bot V - \langle V,\Kapp\rangle_g\partial_s f + \phi \Kapp
 + \langle V,\Kapp\rangle_g \partial_s f= \nabla_{\partial_s}^\bot V + \phi \Kapp.
\end{align*}
The evolution of a vector field $N$ normal to $f$ is given by \eqref{eq:EvoNormal} since from $\langle N,\partial_s f \rangle_g=0 $ and \eqref{eq:EvoUnitVelocity} we get
\begin{align*}
 \nabla_{\partial_t}N
&=  \nabla_{\partial_t}^\bot N + (\partial_t \langle N,\partial_s f \rangle_g - \langle N,\nabla_{\partial_t} \partial_s f \rangle_g)\partial_s f= \nabla_{\partial_t}^\bot N - \langle N, \nabla_{\partial_s}^\bot V + \phi \Kapp\rangle \partial_s f.
\end{align*}

In the next formulas since we have derivatives of second order we expect a contribution from the curvature. By the definition of the Riemannian curvature endomorphism we have
\begin{equation}\label{sera1}
\nabla_{\partial_t}\nabla_{\partial_x}\Phi - \nabla_{\partial_x}\nabla_{\partial_t}\Phi = R(\partial_tf,\partial_xf)\Phi + \nabla_{[\partial_tf, \partial_xf]}\Phi
\end{equation}
and the latter derivative vanishes by \eqref{eq:commutingcoordinates} and linearity. In the case of constant sectional curvature $S_0$ we apply \eqref{eq:RinSpaceForms} to find
\begin{align*}
R(\partial_tf,\partial_xf)\Phi
&= S_0(\langle{\partial_xf,\Phi}\rangle_g \partial_t f - \langle \partial_t f, \Phi\rangle_g \partial_xf)
\\
&= S_0(\langle{\partial_xf,\Phi}\rangle_g V + \phi \langle\partial_x f, \Phi\rangle _g \partial_s f- \langle V, \Phi\rangle_g \partial_xf-\phi \langle\partial_x f, \Phi\rangle _g \partial_s f)\\
&=S_0(\langle{\partial_xf,\Phi}\rangle_g V- \langle V, \Phi\rangle_g \partial_xf).
\end{align*}
that combined with \eqref{sera1} gives \eqref{eq:EvoSCommutators1}. By \eqref{eq:evolutionoflineelement1} and \eqref{eq:EvoSCommutators1} we have
\begin{align*}
 \nabla_{\partial_t}\nabla_{\partial_s}\Phi - \nabla_{\partial_s}\nabla_{\partial_t}\Phi
& = -(\partial_s \phi - \langle V,\Kapp\rangle_g )\nabla_{\partial_s}\Phi + \frac{1}{|\partial_xf|_g} (\nabla_{\partial_t}\nabla_{\partial_x}\Phi - \nabla_{\partial_x}\nabla_{\partial_t}\Phi) \\
&  =- (\partial_s\phi - \langle V,\Kapp\rangle _g)\nabla_{\partial_s} \Phi + \frac{S_0}{|\partial_xf|_g}(\langle{\partial_xf,\Phi}\rangle_g V- \langle V, \Phi\rangle_g \partial_xf),
\end{align*}
which shows \eqref{eq:commutatorS}.

For the normal component of the derivatives we find by compatiblity
\begin{align*}
&\!\!\!\!\nabla_{\partial_t}^\bot \nabla_{\partial_s}^\bot N - \nabla_{\partial_s}^\bot \nabla_{\partial_t}^\bot N\\
&= 
\nabla_{\partial_t} \nabla_{\partial_s}^\bot N 
 - \langle \nabla_{\partial_t} \nabla_{\partial_s}^\bot N ,\partial_s f\rangle_g \partial_s f - \nabla_{\partial_s} \nabla_{\partial_t}^\bot N + \langle \nabla_{\partial_s}\nabla_{\partial_t}^\bot N
,\partial_s f\rangle_g \partial_s f\\
&= 
\nabla_{\partial_t} \nabla_{\partial_s} N 
-\nabla_{\partial_t}(\langle  \nabla_{\partial_s} N  ,\partial_s f\rangle_g \partial_s f)
 - (\partial_t \langle \nabla_{\partial_s}^\bot N ,\partial_s f\rangle_g - \langle  \nabla_{\partial_s}^\bot N ,\nabla_{\partial_t}\partial_s f\rangle_g)\partial_s f\\
 &
\;\;- \nabla_{\partial_s} \nabla_{\partial_t} N 
+\nabla_{\partial_s}(\langle  \nabla_{\partial_t} N  ,\partial_s f\rangle_g \partial_s f)
 + (\partial_s \langle \nabla_{\partial_t}^\bot N ,\partial_s f\rangle_g - \langle  \nabla_{\partial_t}^\bot N ,\Kapp \rangle_g)\partial_s f\\
&= 
\nabla_{\partial_t} \nabla_{\partial_s} N - \nabla_{\partial_s} \nabla_{\partial_t} N\\
& \; \; -\partial_{t}(\langle  \nabla_{\partial_s} N  ,\partial_s f\rangle_g ) \partial_s f
- \langle  \nabla_{\partial_s} N  ,\partial_s f\rangle_g \nabla_{\partial_t}\partial_s f
 + \langle  \nabla_{\partial_s}^\bot N ,\nabla_{\partial_t}\partial_s f\rangle_g \partial_s f\\
 &
\;\;+\partial_s(\langle  \nabla_{\partial_t} N  ,\partial_s f\rangle_g ) \partial_s f + \langle  \nabla_{\partial_t} N  ,\partial_s f\rangle_g \Kapp
 - \langle  \nabla_{\partial_t}^\bot N ,\Kapp \rangle_g\partial_s f\\
 &= 
\nabla_{\partial_t} \nabla_{\partial_s} N - \nabla_{\partial_s} \nabla_{\partial_t} N\\
&\;\; - \langle\nabla_{\partial_t}\nabla_{\partial_s}N, \partial_s f\rangle_g \partial_s f  - \langle\nabla_{\partial_s}N,\nabla_{\partial_t} \partial_s f\rangle_g \partial_s f \\
& \;\;- \langle\nabla_{\partial_s}N, \partial_s f\rangle_g \nabla_{\partial_t} \partial_s f
 + \langle  \nabla_{\partial_s}^\bot N ,\nabla_{\partial_t} \partial_s f \rangle_g\partial_s f\\
  &\;\; +\langle \nabla_{\partial_s}\nabla_{\partial_t}N, \partial_s f \rangle_g \partial_s f  +  \langle  \nabla_{\partial_t}N ,\Kapp \rangle_g\partial_s f + \langle\nabla_{\partial_t}N, \partial_s f\rangle_g \Kapp
 - \langle  \nabla_{\partial_t}^\bot N ,\Kapp \rangle_g\partial_s f
\\&= 
\nabla_{\partial_t} \nabla_{\partial_s} N - \nabla_{\partial_s} \nabla_{\partial_t} N\\
&\;\; + \langle(\nabla_{\partial_s}\nabla_{\partial_t} -\nabla_{\partial_t}\nabla_{\partial_s})N, \partial_s f\rangle_g \partial_s f  - \langle\nabla_{\partial_s}N,\nabla_{\partial_t} \partial_s f\rangle_g \partial_s f  \\
&\;\;- \langle\nabla_{\partial_s}N, \partial_s 
  f\rangle_g\nabla_{\partial_t} \partial_s f
  + \langle  \nabla_{\partial_s}^\bot N ,\nabla_{\partial_t} \partial_s f \rangle_g\partial_s f+ \langle\nabla_{\partial_t}N, \partial_s f\rangle_g \Kapp \, ,
	\end{align*}
	and by \eqref{eq:commutatorS} and \eqref{eq:EvoUnitVelocity} we get
	\begin{align*}
	&\!\!\!\!\nabla_{\partial_t}^\bot \nabla_{\partial_s}^\bot N - \nabla_{\partial_s}^\bot \nabla_{\partial_t}^\bot N\\
&= 
(\langle V,\Kapp\rangle _g-\partial_s\phi )\nabla_{\partial_s }^\bot N - \langle \nabla_{\partial_s}N, \partial_s f \rangle_g  \langle \partial_s f ,\nabla_{\partial_t} \partial_s f\rangle_g \partial_s f 
\\
& \; \; + \langle N, \Kapp \rangle_g \nabla_{\partial_t} \partial_s f - \langle N, \nabla_{\partial_t} \partial_s f \rangle_g \Kapp \\
&= 
(\langle V,\Kapp\rangle _g-\partial_s\phi )\nabla_{\partial_s }^\bot N  + \langle N, \Kapp \rangle_g (\nabla_{\partial_s}^\bot V + \phi \Kapp ) - \langle N, \nabla_{\partial_s}^\bot V + \phi \Kapp \rangle_g \Kapp \\
& = (\langle V,\Kapp\rangle _g-\partial_s\phi )\nabla_{\partial_s }^\bot N  + \langle N, \Kapp \rangle_g \nabla_{\partial_s}^\bot V - \langle N, \nabla_{\partial_s}^\bot V \rangle_g \Kapp \, ,
\end{align*}
where we used that the (space- and time-) derivatives of $\langle N, \partial_s f\rangle_g$ vanish.

It remains to consider the evolution equation of the curvature. Letting $\Phi = \partial_s f$ in \eqref{eq:commutatorS} and applying \eqref{eq:EvoUnitVelocity} we can calculate
\begin{align*}
\nabla_{\partial_t}\Kapp &= \nabla_{\partial_t} \nabla_{\partial_s} \partial_s f  =\nabla_{\partial_s} \nabla_{\partial_t} \partial_s f
- (\partial_s\phi - \langle V,\Kapp\rangle _g)\nabla_{\partial_s}\partial_s f+ S_0 V\\
&= \nabla_{\partial_s}\nabla_{\partial_s}^\bot V+ (\partial_s\phi) \Kapp + \phi \nabla_{\partial_s}\Kapp -
(\partial_s\phi)\Kapp + \langle V,\Kapp\rangle _g\Kapp+ S_0 V\\
&= \nabla_{\partial_s}^\bot\nabla_{\partial_s}^\bot V
+ \langle \nabla_{\partial_s}\nabla_{\partial_s}^\bot V, \partial_s f\rangle_g \partial_s f
 + \phi \nabla_{\partial_s}\Kapp  + \langle V,\Kapp\rangle _g\Kapp+ S_0 V\\
&=(\nabla_{\partial_s}^\bot)^2 V  - 
\langle{\nabla_{\partial_s}^\bot V,\Kapp}\rangle_g \partial_s f + \phi 
\nabla_{\partial_s} \Kapp + \langle V,\Kapp\rangle_g \Kapp + S_0 V,
\intertext{since by compatibility}
\langle \nabla_{\partial_s}\nabla_{\partial_s}^\bot V, \partial_s f\rangle_g &= \partial_s \langle \nabla_{\partial_s}^\bot V, \partial_s f\rangle_g - \langle \nabla_{\partial_s}^\bot V, \nabla_{\partial_s} \partial_s f\rangle_g = - \langle \nabla_{\partial_s}^\bot V, \Kapp \rangle_g.
\end{align*}
This shows \eqref{eq:EvoKappa}. Taking the normal projection we immediately have the subsequent equality \eqref{eq:EvoNormalKappa}. 
\end{proof}

\begin{proof}[{Proof of Lemma \ref{lem:dercurvature}}]
 From \autoref{eq:EvoKappa} we find for $\partial_t f = V + 0\partial_sf = -\nabla\E_\lambda(f)$ from \autoref{eq:GradAsPolynom} that 
 \begin{align*}
  \nabla_{\partial_t}^\bot \Kapp &= (\nabla_{\partial_s}^\bot)^2 V + \langle V,\Kapp \rangle_g \Kapp + S_0 V\\
  &= -(\nabla_{\partial_s}^\bot )^4 \Kapp+ P_3^{2,2}(\Kapp)  + P_1^{2,2}(\Kapp) +  P_5^{0,0}(\Kapp) + P_3^{0,0}(\Kapp)+P_1 ^{0,0}(\Kapp),
 \end{align*}
which shows this lemma for $m=0$. For $m > 0$ we find inductively from \autoref{eq:EvoNormalKappa}
that 
\begin{align*}
\nabla_{\partial_t}^\bot (\nabla_{\partial_s}^\bot)^{m+1}\Kapp &=
\nabla_{\partial_s}^\bot \nabla_{\partial_t}^\bot (\nabla_{\partial_s}^\bot)^m\Kapp 
+ \langle V, \Kapp \rangle \nabla_{\partial_s}^\bot (\nabla_{\partial_s}^\bot)^m \Kapp
\\&\quad + \langle (\nabla_{\partial_s}^\bot)^m \Kapp , \Kapp \rangle  \nabla_{\partial_s}^\bot V - \langle (\nabla_{\partial_s}^\bot)^m \Kapp , \nabla_{\partial_s}^\bot V\rangle \Kapp
\\
&=
\nabla_{\partial_s}^\bot \big{(} 
-(\nabla_{\partial_s}^\bot)^{4+m} \Kapp + P_3^{2+m, 2+m}(\Kapp) + P_1^{2+m,2+m} (\Kapp) \\
&  \qquad + P_5^{m,m}(\Kapp) + P_3^{m,m}(\Kapp) + P_1 ^{m,m}(\Kapp) \big{)}+ P_3^{m+3,\max\{m+1,2\}}(\Kapp)
\\
&\quad  + P_5^{m+1,m+1}(\Kapp) + P_3^{m+1,m+1}(\Kapp) + P_3 ^{m+3,\max\{m,3\}}(\Kapp)\\
&= -(\nabla_{\partial_s}^\bot)^{5+m} \Kapp + P_3^{3+m, 3+m}(\Kapp) + P_1^{3+m,3+m}(\Kapp)\\
   &\quad + P_5^{m+1,m+1}(\Kapp) + P_3^{m+1,m+1}(\Kapp) + P_1 ^{m+1,m+1}(\Kapp) .\qedhere
\end{align*}
\end{proof}

\begin{proof}[{Proof of Lemma \ref{lem:dercurvnormnicht}}]
 The first assertion follows from \autoref{eq:PartialIntegration} since
 \[\nabla_{\partial_s} \Kapp = \nabla_{\partial_s }^\bot\Kapp + \langle \nabla_{\partial_s} \Kapp , \partial_s f\rangle  \partial_s f =  \nabla_{\partial_s }^\bot\Kapp+0 - |\Kapp|^2 \partial_s f.\]
Let us show the second statement inductively. For $m =2$ we find
 \begin{align*}
  \nabla_{\partial_s}^2 \Kapp &=\nabla_{\partial_s} \nabla_{\partial_s} ^\bot\Kapp -\partial_s(|\Kapp|^2) \partial_s f - |\Kapp|^2 \nabla_{\partial_s } \partial f\\
  &= (\nabla_{\partial_s}^\bot)^2\Kapp + \langle \nabla_{\partial_s} \nabla_{\partial_s}^\bot \Kapp, \partial_s f\rangle \partial_s f - 2 \langle {\Kapp, \nabla_{\partial_s} ^\bot}\Kapp\rangle \partial_s f - |\Kapp|^2 \Kapp\\
  &=  (\nabla_{\partial_s}^\bot)^2\Kapp -3 \langle \nabla_{\partial_s}^\bot  \Kapp, \Kapp \rangle \partial_s f - |\Kapp|^2 \Kapp =   (\nabla_{\partial_s}^\bot)^2\Kapp -P_{2}^{1,1} (\Kapp) \partial_s f - P_3^{0,0}(\Kapp),
 \end{align*}
and for the induction step we first note that for $b$ even we have by convention that
$\nabla_{\partial_s} P_b^{a,c}(\Kapp) = \partial_s P_b^{a,c}(\Kapp) =P_b^{a+1,c+1}(\Kapp)$, while for odd $b$ we find that
\begin{align*}
 \nabla_{\partial_s}P_b^{a,c} &= \nabla_{\partial_s}^\bot P_b^{a,c} + \langle \nabla_{\partial_s}P_b^{a,c}, \partial_s f\rangle \partial_s f = P_b^{a+1,c+1} + 0 - \langle P_b^{a,c}, \nabla_s \partial_s f\rangle \partial_s f \\
 &= P_b^{a+1,c+1}+ P_{b+1}^{a,c} \partial_sf.
\end{align*}
We thus have
\begin{align*}
 \nabla_{\partial_s}^m \Kapp &= \nabla_s \big ((\nabla_{\partial_s }^\bot)^{m-1}\Kapp 
  + \sum_{b=2,\; b\text{ even}}^{m} P_b^{m -b, m-b}(\Kapp) \partial_sf 
  +  \sum_{b=3,\; b\text{ odd}}^{m} P_b^{m -b, m-b}(\Kapp) \big)\\
  &= (\nabla_{\partial_s}^\bot)^m \Kapp + \langle \nabla_{\partial_s } (\nabla_{\partial_s}^\bot)^{m-1}\Kapp, \partial_s f\rangle \partial_s f\\
  &\quad + \sum_{b=2,\; b\text{ even}}^{m} (P_b^{m +1  - b, m+1 -b}(\Kapp) \partial_sf + P_b^{m -b, m-b}(\Kapp) \Kapp)\\
  & \quad + \sum_{b=3,\; b\text{ odd}}^{m}(P_b^{m+1 -b, m+1-b}(\Kapp) + P_{b+1}^{m -b, m-b}(\Kapp)\partial_s f)\\
  &= (\nabla_{\partial_s}^\bot)^m \Kapp +P_2^{m-1,m-1}(\Kapp) \partial_s f\\
  & \quad + \sum_{b=2,\; b\text{ even}}^{m} P_b^{m +1  - b, m+1 -b}(\Kapp) \partial_sf
   + \sum_{b=2,\; b\text{ even}}^{m} P_{b+1}^{m -b, m-b}(\Kapp)\\
  & \quad + \sum_{b=3,\; b\text{ odd}}^{m} P_b^{m+1 -b, m+1-b}(\Kapp) +\sum_{b=4,\; b\text{ even}}^{m+1}
  P_{b}^{m -(b-1), m-(b-1)}(\Kapp)\partial_s f\\
  & = (\nabla_{\partial_s }^\bot)^m\Kapp 
  + \sum_{b=2,\; b\text{ even}}^{m+1} P_b^{m+1 -b, m+1-b}(\Kapp) \partial_sf 
  +  \sum_{b=3,\; b\text{ odd}}^{m+1} P_b^{m+1 -b, m+1-b}(\Kapp).\qedhere
\end{align*}

\end{proof}

\begin{proof}[{Proof of Lemma \ref{lem:controlparsuff}}]
Here we use repeatedly that $\nabla_{\partial_x}=  \gamma \nabla_{\partial_s}$ since $\gamma = |\partial_x f|_g$. Then for $m=1$ we have $ \nabla_{\partial_x} N = \gamma \nabla_{\partial_s} N$. By characterisation of the Levi-Civita connection 
$$ \nabla_{\partial_x}^2 N =  \nabla_{\partial_x} (\gamma \nabla_{\partial_s} N) = \gamma^2 \nabla_{\partial_s}^2 N + (\partial_x \gamma) 
\nabla_{\partial_s} N \, .$$
The general statement follows then by induction. Indeed,
\begin{align*}
\nabla_{\partial_x}^{m+1} N & =  \nabla_{\partial_x} (\gamma^m   \nabla_{\partial_s}^m N + \sum_{j=1}^{m-1} P_{m,j}(\gamma, ....,\partial_x^{m-j} \gamma ) \nabla_{\partial_s}^j N) \\
& = \gamma^{m+1}   \nabla_{\partial_s}^{m+1} N  + m \gamma^{m-1} (\partial_x \gamma) \nabla_{\partial_s}^m N\\
& + \sum_{j=1}^{m-1} \gamma P_{m,j}(\gamma, ....,\partial_x^{m-j} \gamma ) \nabla_{\partial_s}^{j+1} N 
+ \sum_{j=1}^{m-1} \big( \partial_x P_{m,j}(\gamma, ....,\partial_x^{m-j} \gamma ) \big) \nabla_{\partial_s}^j N \\
& = \gamma^{m+1}   \nabla_{\partial_s}^{m+1} N  + \sum_{j=1}^{m} P_{m+1,j}(\gamma, ....,\partial_x^{m+1-j} \gamma ) \nabla_{\partial_s}^j N
\end{align*}
mit $P_{m+1,j}$ polynomials of degree at most $m$, being the $P_{m,j}$ polynomials of degree at most $m-1$.
\end{proof}

\begin{proof}[Proof of \eqref{eq:sabbath}]\label{sabbath}
We first rewrite some formulas using the notation with the $P^{a,c}_b(\Kapp)$. From \eqref{VV}, \eqref{eq:EvoNormalKappa}
\begin{align*}\label{VV}
\nabla_{\partial_s}^{\bot} V & =  P_1^{3,3}(\Kapp) + P_3^{1,1}(\Kapp) +P_1^{1,1}(\Kapp) \, ,\\
\nabla_{\partial_t}^{\bot} \Kapp & = P_1^{4,4}(\Kapp) + P_3^{2,2}(\Kapp) +P_1^{2,2}(\Kapp) + P_5^{0,0}(\Kapp) +P_3^{0,0}(\Kapp) +P_1^{0,0}(\Kapp) \, ,
\end{align*} 
and hence
$$ \nabla_{\partial_t}^{\bot}  P^{0,0}_b(\Kapp) =  P_b^{4,4}(\Kapp) +P_{b+2}^{2,2}(\Kapp) + P_{b}^{2,2}(\Kapp) + P_{b+4}^{0,0}(\Kapp) +P_{b+2}^{0,0}(\Kapp) +P_{b}^{0,0}(\Kapp) \,  .$$
We compute then with \eqref{eq:Commutator_Normal_Normal_S}
\begin{align*}
\nabla_{\partial_t}^{\bot} V &  = - \nabla_{\partial_s}^{\bot}  (\nabla_{\partial_t}^{\bot}\nabla_{\partial_s}^{\bot} \Kapp)  + 
(P_2^{2,2}(\Kapp) + P_4^{0,0}(\Kapp) +P_2^{0,0}(\Kapp)) (\nabla_{\partial_s}^\bot)^2 \Kapp \\
& \quad \quad + P^{1,1}_2(\Kapp) (P_1^{3,3}(\Kapp) + P_3^{1,1}(\Kapp) +P_1^{1,1}(\Kapp) ) +  P_3^{4,4}(\Kapp) + P_5^{2,2}(\Kapp) +P_3^{2,2}(\Kapp) \\
& \quad \quad + P^{0,0}_7(\Kapp) +  P^{0,0}_5(\Kapp)  + P^{0,0}_3(\Kapp) + P^{4,4}_1(\Kapp) +  P^{2,2}_1(\Kapp)  + P^{0,0}_1(\Kapp)\\
& =  (\nabla_{\partial_s}^{\bot})^2  (P_1^{4,4}(\Kapp) + P_3^{2,2}(\Kapp) +P_1^{2,2}(\Kapp) + P_5^{0,0}(\Kapp) +P_3^{0,0}(\Kapp) +P_1^{0,0}(\Kapp) ) \\
& \quad \quad + \nabla_{\partial_s}^{\bot} ( P_5^{1,1}(\Kapp) +P_3^{1,1}(\Kapp) + P_3^{3,3}(\Kapp)) \\
& \quad \quad + P^{0,0}_7(\Kapp) +  P^{0,0}_5(\Kapp)  + P^{0,0}_3(\Kapp) + P^{4,4}_1(\Kapp) +  P^{2,2}_1(\Kapp)  + P^{0,0}_1(\Kapp)\\
& \quad \quad + P_3^{4,4}(\Kapp) + P_5^{2,2}(\Kapp) +P_3^{2,2}(\Kapp) \\
& =  P_1^{6,6}(\Kapp) + P_3^{4,4}(\Kapp) +P_1^{4,4}(\Kapp) + P_5^{2,2}(\Kapp) +P_3^{2,2}(\Kapp) +P_1^{2,2}(\Kapp) \\
& \quad \quad + P^{0,0}_7(\Kapp) +  P^{0,0}_5(\Kapp)  + P^{0,0}_3(\Kapp) + P^{0,0}_1(\Kapp)\, .
\end{align*}

\end{proof}
\section{Details for the Short Time Existence}
\label{sec:STE}
\subsection{Proof of \autoref{prop:NormalGraph}}
\label{proof:lemma_normal_graph}


The main ingredient of the proof of \autoref{prop:NormalGraph} is the implicit function theorem, for which we need a lower bound for the radius of the domain of the implicit function, for which we could not find an adequate reference in the literature.

\subsubsection{A  control of the domain of the implicit function}


\begin{lemma}\label{lemma:inverse_function_lower_bound}
 Let $X,Y$ be Banach spaces, $U \subset X$ open and $f\in \CalC^1(U;Y)$ such that $f'(a)$ is an isomorphism for some $a\in U$. If there exists some $r > 0$ such that $\overline {B}_r(a) \subset U$ and
 \[
  \|f'(x)-f'(z)\| \leq \frac{1}{2 \|f'(a)^{-1}\|}\quad \text{for all }x,z \in \overline {B}_r(a),
 \]
 then for all $\tilde y \in \displaystyle \overline{B}_{\frac{r}{2 \|f'(a)^{-1}\|}}\left(f(a)\right) \subset Y$ there exists a unique $\tilde x \in \overline {B}_r(a)$ with $f(\tilde x) = \tilde y$.
\end{lemma}

\begin{proof}
Let $\tilde f(x) = (f'(a))^{-1}(f(x+a)-f(a))$ and $s = \frac{1}{2}$, then this lemma follows from \cite[Chapter 6, \S 1, Lemma 1.3]{Lang_Analysis}. 
\end{proof}

With this lemma we gain a control from below on the radius of the implicit function's domain.

\begin{prop}\label{prop:lower_bound_implicit}
 Let $X,Y,Z$ be Banach spaces, $U \subset X, V \subset Y$ open, $F\in \CalC^1(U\times V;Z)$, $a = (a_1,a_2) \in U \times V$ with $F(a) = 0$ and $D_2F(a)\colon Y \to Z$ be an isomorphism. Let $\lambda = \frac{1}{2}\left( 1+\|D_2F(a)^{-1}\|(1+ \|D_1F(a)\|)\right)^{-1}$ and $r>0$ such that $\overline{B}_r(a) \subset (U \times V)$ and 
 \[
 \|DF(x,y) -DF(\tilde x, \tilde y)\| \leq \lambda \quad \text{ for all } (x,y), (\tilde x, \tilde y) \in \overline{B}_r(a).
 \]
 Then for all $x \in U$ with $\|x-a_1\| \leq \lambda r$ there exists a unique $y \in V$ with $\|y-a_2\| \leq r$ and $F(x,y) = 0$.\\
 Moreover, if $F$ is twice continuously differentiable with $\|D^2F\| \leq M$ we can choose 
 \[r = \frac{1}{4 M} \left( 1+\|D_2F(a)^{-1}\|(1+ \|D_1F(a)\|)\right)^{-1}.\]
\end{prop}

\begin{proof}
 The function $f \colon U \times V \to X \times Z$, $(x,y) \mapsto (x,F(x,y))$ is continuously differentiable and satisfies
 \[
  Df(x,y) = \begin{pmatrix}
             \operatorname {id}_X& 0 \\ D_1F(x,y) & D_2F(x,y)
            \end{pmatrix},
 \]
 thus $Df(a)$ is an isomorphism with 
 \[
    Df(a)^{-1} = \begin{pmatrix}
             \operatorname {id}_X& 0 \\ -D_2F(a)^{-1}D_1F(a) & D_2F(a)^{-1}
            \end{pmatrix},
 \]
 and $\|Df(a)^{-1}\| \leq 1 + \|D_2F(a)^{-1}\|\|D_1F(a)\| + \|D_2F(a)^{-1}\| = \frac{1}{2\lambda}$. Thus, for all $(x,y), (\tilde x, \tilde y) \in \overline{B}_r(a)$ we find
 \[
  \|Df(x,y) - Df(\tilde x, \tilde y) \| =   \|DF(x,y) - DF(\tilde x, \tilde y) \| \leq \lambda \leq \frac{1}{2\|Df(a)^{-1}\|}.
 \]
 Let $\|x-a_1\|\leq \lambda r$. Then $\omega = (x,0) \in \overline{B}_{\lambda r}(f(a))$ and thus from \autoref{lemma:inverse_function_lower_bound} we find some unique $(\hat x, y) \in \overline{B}_{r}(a)$ with $(x,0) = \omega = f(\hat x, y) = (\hat x, F(\hat x, y))$, from which the first part follows. The second part follows immediately from the mean value theorem.
 \end{proof}

\subsubsection{Proof of \autoref{prop:NormalGraph}}
As an application of \autoref{prop:lower_bound_implicit} we can show the following lemma.

\begin{lemma}\label{lemma:normalgraph} 
 Let $m \in \N_0$ and $\overline f \in \CalC^{5+m,\alpha}(\Sph^1;\Hyp^2)$ be an immersion. Then there exists some constant $\sigma(\|\overline f\|_{\CalC^{4,\alpha}}, \min \overline f_2, \min |\partial_x \overline f|) > 0$ such that for all $\psi \in \CalC^{4+m,\alpha}(\Sph^1;\Hyp^2)$ with $\|\psi\|_{ \CalC^{4,\alpha}} \leq \sigma$ the function $\overline f + \psi $ (where the addition is defined within the global chart from \autoref{subs:hyp}) is an immersion of $\Sph^1 \to \Hyp^2$, and there exists a unique diffeomorphism $\Phi$ of $\Sph^1$ such that 
 \[
   (\overline f + \psi) \circ \Phi = \overline f +  N,
 \]
 where the function $N\in \CalC^{4+m,\alpha}(\Sph^1;\Hyp^2)$ is orthogonal to $\partial_x\overline f$.
\end{lemma}
\begin{proof} We begin with $m=0$ and identify $\Sph^1 \cong \R/(2\pi\Z)$ and lift $\overline f$ to a periodic function $\tilde f\colon \R \to \Hyp^2$ to obtain a linear structure. We denote the subspaces of $2\pi$-periodic function with the index \textit{per}.
 \\
 \emph{Claim}: Let $\tilde f \in \CalC^{5,\alpha}_{per}(\R;\R^2)$ be an immersion with $\tilde f_2 > 0$, then there exists a $\sigma >0$ such that for all $\tilde \psi \in  \CalC^{4,\alpha}_{per}(\R;\R^2)$ with $\|\tilde \psi \|_{ \CalC^{4,\alpha}} \leq \sigma$ there exists a $\tilde \Phi \in \CalC^{4,\alpha}(\R;\R)$ which is strictly increasing and satisfies $\tilde \Phi(x + 2\pi) = \tilde \Phi(x) + 2\pi$ for all $x \in \R$, such that
  $(\tilde f + \tilde \psi )\circ \tilde \Phi = \tilde f + \tilde N$
 for some $\tilde N \in \CalC^{4,\alpha}_{per}(\R;\R^2)$ with $\langle \tilde N, \partial_x \tilde f \rangle_{g(\tilde f)} = 0$. 
 Furthermore, $\sigma$ depends only on $c_1 \defeq \|\tilde f\|_{\CalC^{4,\alpha}}$, $c_2 \defeq \min \tilde f_2$ and $c_3 \defeq \min |\partial_x \tilde f|$.\\
 It is clear that the claim implies \autoref{lemma:normalgraph} for $m=0$. Let us show the claim. Let $\tilde f$ be given and define $X \defeq \CalC^{4,\alpha}_{per}(\R;\R^2)$, $Y \defeq Z \defeq \CalC^{2,\alpha}_{per}(\R;\R)$. Let $\rho_1 \defeq \frac{1}{2}\min\{c_2,c_3\}$ and $\rho_2 \defeq \frac{1}{2}$. Then for all $\tilde \psi \in B_{\rho_1}(0) \subset X$ the function $\tilde g \defeq \tilde f + \tilde \psi$ satisfies $\partial_x \tilde g \neq 0$ and $\tilde g_2 > 0$. Moreover, for all $\tilde \phi \in B_{\rho_2}(0) \subset Y$ the function $\tilde \Phi \defeq \operatorname{id} + \tilde \phi \colon \R \to \R$ is strictly increasing and satisfies $\tilde \Phi(x+2\pi) = \tilde \Phi(x)+2\pi$ for all $x \in \R$. Let us first show that we can choose some possibly smaller $\rho_1(c_1,c_2,c_3)>0$ such that
 \begin{equation}
  \label{eq:NonZero}
  \langle \partial_x ( \tilde f + \tilde \psi )\circ (\operatorname{id} + \tilde \phi),\partial_x \tilde f \rangle_{g(\tilde f)} \geq \frac{1}{2} \min |\partial_x\tilde f|^2_{g(\tilde f)}
 \end{equation}
 for all $\tilde \psi \in B_{\rho_1}(0) \subset X, \tilde \phi \in B_{\rho_{2}}(0) \subset Y$. Indeed, by Cauchy-Schwarz we find
\begin{align*}
 &|\langle \partial_x ( \tilde f + \tilde \psi )\circ (\operatorname{id} + \tilde \phi), \partial_x \tilde f \rangle_{g(\tilde f)} -  \min |\partial_x\tilde f|^2_{g(\tilde f)}| \\
 & \qquad \leq \left( |\partial_x \tilde f \circ (\operatorname {id} + \tilde \phi) - \partial_x \tilde f|_{g(\tilde f)}  + |\partial_ x \tilde \psi \circ (\operatorname{id}+\tilde \phi) |_{g(\tilde f)}\right) |\partial_x \tilde f|_{g(\tilde f)}
 \\
 &\qquad \leq c(c_2) \left( \|\partial_x^2 \tilde f\|_\infty \|\tilde \phi\|_\infty + \|\partial_x \tilde \psi\|_\infty \right) \|\partial_x \tilde f\|_\infty
 \\
 &\qquad \leq c(c_2) \|\tilde f \|_{\CalC^1} (\|\tilde f \|_{\CalC^2} +1)(\|\tilde \phi \|_{\CalC^0} + \|\tilde \psi \|_{\CalC^1}) \leq \frac{1}{2} \min |\partial_x \tilde f|_{g(\tilde f)}.
\end{align*}
 for $\rho_1(c_1,c_2,c_3) > 0$ small enough. Let $U \defeq  B_{\rho_1}(0) \subset X$, $V\defeq B_{\rho_2}(0) \subset Y$ and
 \[
  F\colon U \times V \to Z,\quad (\tilde \psi, \tilde \phi) \mapsto 
  \langle ( \tilde f + \tilde \psi )\circ (\operatorname{id} + \tilde \phi) - \tilde f,  \partial_x \tilde  f \rangle_{g(\tilde f)}.
 \]
 Then $F$ is well defined and $\CalC^2$ with
\[D_1 F(\tilde \psi, \tilde \phi) \tilde h_1 = \langle \tilde h_1 \circ (\operatorname{id}+\tilde \phi),\partial_x \tilde f \rangle _{g(\tilde f)}, \qquad  D_2 F(\tilde \psi, \tilde \phi) \tilde h_2 = \tilde h_2 \langle \partial_x (\tilde f +\tilde \psi) \circ (\operatorname{id}+\tilde \phi),\partial_x \tilde f \rangle _{g(\tilde f)},
\]and
\begin{align*}&D_{11} F(\tilde \psi, \tilde \phi) [\tilde h_1, \hat h_1] = 0, \qquad D_{12} F(\tilde \psi, \tilde \phi) [\tilde h_1, \tilde h_2] = \tilde h_2 \langle \partial_x \tilde h_1 \circ (\operatorname{id}+\tilde \phi),\partial_x \tilde f \rangle _{g(\tilde f)},\\
   &D_{22} F(\tilde \psi, \tilde \phi) [\tilde h_2,\hat h_2] = \tilde h_2 \hat h_2 \langle \partial_x^2  (\tilde f +\tilde \psi) \circ (\operatorname{id}+\tilde \phi),\partial_x \tilde f \rangle _{g(\tilde f)},
 \end{align*}
from which one can show (using $\|u\cdot v\|_{\CalC^{2,\alpha}} \leq C(\Omega,\alpha)\|u\|_{\CalC^{2,\alpha}} \|v\|_{\CalC^{2,\alpha}}$ and $\|u\circ v\|_{\CalC^{2,\alpha}} \leq \|u\|_\infty + C(\Omega,\alpha, \tilde c)\|u\|_{\CalC^{2,\alpha}} \|v\|_{\CalC^{2,\alpha}}$ for all $\|v\|_{\CalC^{2,\alpha}}\leq \tilde c$)
that
\begin{equation}
 \label{eq:D^2FM}
 \|D^2F(\tilde \psi, \tilde \phi)\| \leq C(2\pi, \alpha, \rho_1, \rho_2, c_1, c_2)  \eqdef M
\end{equation}
for some $M = M (2\pi, \alpha, c_1, c_2, c_3) > 0$. Moreover we find for $\tilde \phi = 0$, $\tilde \psi = 0$ that
 \begin{equation}
 \label{eq:D1F}
 \|D_1F(0,0)\| \leq C(2\pi, \alpha, c_1,c_2)
\end{equation}
and since $D_2F(0,0)\tilde h_2 = \tilde h_2 \cdot |\partial_x \tilde f|^2_{g(\tilde f)}$ we see that $D_2F(0,0) \colon Y \to Z$ is invertible with 
 \begin{equation}
 \label{eq:D2F}
 \|D_2F(0,0)^{-1}\| \leq C\left(\alpha, 2\pi, \| \left(|\partial_x \tilde f|_{g(\tilde f)}\right)^{-1}\|_{\CalC^{2,\alpha}}\right) \leq C( 2\pi,\alpha, c_1,c_2, c_3).
 \end{equation}
 Let $\lambda = \frac{1}{2}\left( 1+\|D_2F(0)^{-1}\|(1+ \|D_1F(0)\|)\right)^{-1}$, $r=\frac{\lambda}{2M}$ 
 and define $\sigma \defeq \min\{\rho_2, \lambda r\}$. 
 We find that $\sigma$ only depends on the constants $c_1$, $c_2$ and $c_3$. 
 Now \autoref{prop:lower_bound_implicit} shows that for all $\|\tilde \psi\|_{\CalC^{4,\alpha}} \leq \sigma$ there exists a unique $\tilde \phi \in \CalC^{2,\alpha}_{per}(\R;\R)$ with $\|\tilde \phi\| \leq \min\{r, \rho_2\}$ such that $F(\tilde \psi,\tilde \phi) = 0$. 
 To finish the proof of the claim it remains to show that $\tilde \phi \in \CalC^{4,\alpha}$. Indeed, if we differentiate the equation $0 = F(\tilde \psi,\tilde \phi)$ we find using \autoref{eq:NonZero} that
 \begin{align}
  &1+ \partial_x \tilde \phi = \left(\langle \partial_x (\tilde f + \tilde \psi)\circ (\operatorname{id} + \tilde  \phi), \partial_x \tilde f \rangle _{g(\tilde f)} \right)^{-1}\cdot  {\Big{(}}
  2  \frac{\partial_x \tilde f_2}{\tilde f_2} \langle ( \tilde f + \tilde \psi )\circ (\operatorname{id} + \tilde \phi) - \tilde f,  \partial_x \tilde  f \rangle_{g(\tilde f)}
      \nonumber\\
      &\qquad\qquad \qquad 
      +
      |\partial_x \tilde f|_{g(\tilde f)}^2 - \langle (\tilde f + \tilde \psi)\circ(\operatorname{id}+\tilde \phi)-\tilde f, \partial_x^2 \tilde f , \rangle_{g(\tilde f)}
  {\Big{)}},\label{eq:diffeo_is_smoother}
 \end{align}
which shows that $\partial_x \tilde \phi \in \CalC^{2,\alpha}$, i.e. $\tilde \phi \in \CalC^{3,\alpha}$ and hence, using this equation again and the fact that $\tilde f \in \CalC^{5,\alpha}$ we see that $\partial_x \phi \in \CalC^{3,\alpha}$, which finishes the proof for $m=0$. The case of $m\geq 1$ follows similarly from \eqref{eq:diffeo_is_smoother}.
\end{proof}

\begin{proof}[{Proof of \autoref{prop:NormalGraph}}]
Let $\mu_1 \defeq \frac{1}{2}\min \{\min (f_0)_2, \min_{x} |\partial_xf_0|\}$. 
Then any $h_1, h_2 \colon \Sph^1 \to \Hyp^2$ satisfying $\|h_i-f_0 \|_{\CalC^{4+m,\alpha}} \leq \mu_1$ is immersed.  For any vector field $X$ along $h_1$ we have $\langle X, \partial_x h_1 \rangle_{g(h_1)}\partial_x h_1 - \langle X, \partial_x  h_1\rangle_{g(h_1)}\partial_x  h_1  = 0$, whence there exists some  $\mu_2 > 0$ small enough independent of $X$ such that for $\|h_1 - h_2 \|_{\CalC^1} \leq 2\mu_2$ we have after translating $X$ along $h_2$
    \[|\langle X(x), \partial_x h_1(x) \rangle_{g(h_1(x))}\partial_x h_1(x) - \langle X(x), \partial_x h_2(x)\rangle_{g(h_2(x))}\partial_x h_2(x)|_{g(h_2(x))} \leq \frac{1}{2} |X(x)|_{g(h_2(x))}.\] 
    Thus, 
    with $X =  N \bot \partial_x h_1$ we find for these $h_1,h_2$ that
    \begin{align*}
      &|\Pi_{h_2(x)}^\bot N(x)|_{g(h_2(x))} = |N(x) - \Pi_{h_2(x)}^\top N(x) + \Pi_{h_1(x)}^\top N(x)|_{g(h_2(x))} \\
      &\qquad \geq |N(x)|_{g(h_2(x))} - 
      |\langle N, \partial_x h_1 \rangle_{g(h_1)}\partial_x h_1 - \langle N, \partial_x h_2\rangle_{g(h_2)}\partial_x h_2|_{g(h_2)}\\
      &\qquad 
      \geq  \frac{1}{2} \sup _x |N(x)|_{g(h_2(x))} > 0,
    \end{align*}
    which shows the first part of \autoref{prop:NormalGraph} for any $\tilde f$ satisfying $\|\tilde f-f_0\|_{\CalC^{4+m,\alpha}} \leq \tilde \mu \defeq \frac{1}{2} \min\{\mu_1,\mu_2\}$ and any $0<\mu \leq \tilde \mu$.
\\Let $\delta \defeq \min\{ \frac{1}{2} \min (f_0)_2, \mu_1, \frac{1}{2} \|f_0\|_{\CalC^{4,\alpha}}\}$. Then $\delta > 0$ and for any $\tilde f\in \CalC^{5+m,\alpha}(\Sph^1;\Hyp^2)$ with $\|\tilde f-f_0\|_{\CalC^{4,\alpha}} \leq \delta$ we find that $\tilde f$ is an immersion and satisfies
\[ \delta \leq \frac{1}{2} \min (f_0)_2  \leq \min \tilde f_2 \leq \delta + \|f_0\|_{\CalC^{4,\alpha}},
\quad  \delta \leq \mu_1 = \frac{1}{2}\min |\partial_x\tilde f_0| \leq \min |\partial_xf| \leq \delta + \|f_0\|_{\CalC^{4,\alpha}}
\]
and $  \delta \leq \frac{1}{2}\|f_0\|_{\CalC^{4,\alpha}} \leq \|\tilde f\|_{\CalC^{4,\alpha}} \leq \delta +  \|f_0\|_{\CalC^{4,\alpha}}$.
From \autoref{lemma:normalgraph} we see that there exists some $\tilde \sigma(\| \tilde f\|_{\CalC^{4,\alpha}}, \min \tilde f_2, \min |\partial_x \tilde f|) > 0$ such that for all $\psi \in \CalC^{4+m,\alpha}(\Sph^1;\Hyp^2)$ with $\|\psi\|_{ \CalC^{4,\alpha}} \leq \tilde  \sigma$ there exists a unique diffeomorphism $\Phi$ of $\Sph^1$ such that $(\tilde  f + \psi) \circ \Phi =  \tilde f + u \tilde N$, where $\tilde N$ is a smooth unit normal vector field along $\tilde f$ and $ u \in \CalC^{4+m,\alpha}(\Sph^1;\R)$ is a function. Whence we find that $\tilde \sigma \geq \sigma(\delta, \|f_0\|_{\CalC^{4,\alpha}}) > 0$. 
Now put $\mu \defeq \min\{\tilde \mu, \frac{\sigma}{2}\}$. By density we can now choose a smooth function $\overline f \colon \Sph^1 \to \Hyp^2$ satisfying $\|\overline f - f_0\|_{\CalC^{4,\alpha}} \leq  \min\{\mu, \delta\}$. From the construction we then find that $\overline f$ satisfies all assertions of \autoref{prop:NormalGraph}.\qedhere
\end{proof}

\subsection{Schauder Theory}
\subsubsection{Proof of \autoref{thm:linearIsoSchauderM}}
\label{proof:ProofLinearSchauderM}
To show \autoref{thm:linearIsoSchauderM} we recall the classic Schauder results for parabolic equations of fourth order 
on intervals, where we follow \cite[VI.3]{Eidelman}. Thus we consider linear parabolic problems of the type
\begin{equation} \label{eq:schauder_linear}
 \left \{
\begin{array}{rll}
 \partial_t u(x,t) - \sum_{|\gamma|\leq 4} a_\gamma (x,t) \partial_x^\gamma u &= f(x,t), &(x,t) \in I \times [0,T],\\
 u(0,x) &= u_0(x)&x \in I\\
 u|_S = 0, \partial_x u|_S &= 0,
\end{array}
\right.
\end{equation}
where $ I = (a,b) \subset \R$ is an open bounded interval. Here the functions $f:\overline {Q_T} \defeq \overline I \times [0,T] \to \R$, $u_0 \colon \overline I \to \R$, $\phi_1, \phi_2 \colon S \defeq \partial I \times [0,T] \to \R$ are given, and $u \colon \overline {Q_T} \to \R$ is the unknown. Note that the boundary condition satisfies the Shapiro-Lopatinski\u{\i} condition (see \cite[Definition I.8]{Eidelman}). We fix some $s>0, s \notin \N$. 
%

\begin{mythm}[{\cite[Theorem VI.21]{Eidelman}}]
Assume that \label{thm:linearIsoSchauder} the coefficients of the equation satisfy $a_\gamma \in \HS[s]{\frac{s}{4}}(\overline {Q_T})$, the data satisfies the smoothness assumptions $f \in \HS[s]{\frac{s}{4}}(\overline {Q_T})$, 
$u_0 \in \CalC^{4+s}(\overline I)$ and the compatibility conditions of order $\lfloor \frac{s + 4}{4}\rfloor$ (c.f. \cite[p. 319f]{MR0241821} or \cite[p. 219]{Eidelman}). Moreover, we assume that $L$ is parabolic in the sense of Petrovskii (c.f. \cite[Def. I.1]{Eidelman}).
Then there exists some constant $C>0$ independent of $f$ and $u_0$ and some unique solution $u \in \HS[4+s]{\frac{4+s}{4}} 
(\overline {Q_T})$ to the problem \eqref{eq:schauder_linear}. The solution additionally satisfies
\begin{equation}
\label{eq:norm_iso_L_schauder}
\| u \| _ {\HS[4+s]{\frac{4+s}{4}}(  
\overline {Q_T}
)} \leq C (\|f\|_ {\HS[s]{\frac{s}{4}} (\overline {Q_T}
)} + \| u_0 \| _ {\CalC^{4+s}(\overline I
)}).
\end{equation}
\end{mythm}

We can now show \autoref{thm:linearIsoSchauderM}.

\begin{proof} [{Proof of \autoref{thm:linearIsoSchauderM}}]
 Uniqueness follows from Young's and Gr\"onwall's inequality and an approximation argument. 
 \\To show existence we want to apply \autoref{thm:linearIsoSchauder}, so we need to work with coordinates. As explained in \eqref{eq:charts} we choose  four charts $\phi_i \colon U_i \to V_i$ of $\Sph^1 = \{(x,y) \in \R^2 \st x^2 + y^2 = (2\pi)^{-2}\}$, where $U_i = (0,\frac{1}{2})$ and $V_i$ is the intersection of $\Sph^1$ with the canonical {half planes} in $\R^2$, such that $\phi_i$ is an isometry (between Riemannian manifolds) for all $i=1,\ldots,4$. Then the geodesic distance for $(x,y)\in V_i$ is given by $d(x,y) = |\phi^{-1}(x)-\phi^{-1}(y)|_{\R}$. 
 Let $\tilde \phi_i \colon \tilde U_i \defeq  U_i \times [0,T] \to \tilde V_i = V_i \times [0,T] \subset \Sph^1 \times [0,T]$ be the corresponding charts on $\Sph^1 \times [0,T]$. Again we find that these charts are isometric.\\
 We choose a smooth partition of unity $1 = \chi_1 + \chi_2+\chi_3+\chi_4$ on $\Sph^1$ with $\operatorname{supp} \chi_i \subset V_i$. On $\tilde U_i = (0,\frac{1}{2})\times [0,T]$ we consider the equation
\begin{equation} 
 \left \{
\begin{array}{rl}
  ((\tilde \phi_i^{-1})^*L)u &= (f\chi_i)\circ \tilde \phi_i,\\
 u(0,\cdot) &= (\chi_i u_0)\circ  \tilde \phi_i\\
 u|_{\partial U_i} =0, \partial_x u|_{\partial U_i} &= 0, \end{array}
\right.
\end{equation}
where $(\tilde \phi_i^{-1})^*L$ denotes the operator $L$, written in local coordinates on $\tilde \phi_i^{-1}( \tilde V_i) = (0,\frac{1}{2}) \times [0,T]$, that is:  $(\tilde \phi_i^{-1})^*L = \partial_t - \sum a_\gamma\circ \tilde \phi_i\partial_x^\gamma$.
Since $ \chi_i \circ \phi_i$ has compact support in $U_i$ we find that all compatibility conditions are satisfied. \autoref{thm:linearIsoSchauder} gives the existence of a unique solution $u_i \colon \tilde U_i \to \R$ and some $C_i$ such that 
\begin{align}
\| u_i \| _ {\HS[4+s]{\frac{1+s}{4}} 
( \overline { U}_i \times [0,T])
}
&\leq C_i (\|(\chi_i f)\circ \tilde \phi_i \|_ {\HS[s]{\frac{s}{4}} ( \overline { U}_i \times [0,T]) } + \|(\chi_i u_0)\circ \phi_i \| _ {\CalC^{4+s}( \overline { U}_i)})\nonumber \\
&= C_i (\|\chi_i f\|_ {\HS[s]{\frac{s}{4}} ( \overline { V}_i \times [0,T]) } + \|\chi_i u_0 \| _ {\CalC^{4+s}( \overline { V}_i)})\nonumber \\
&\leq C (\|f\|_ {\HS[s]{\frac{s}{4}} ( \Sph^1 \times [0,T]) } + \| u_0 \| _ {\CalC^{4+s}( \Sph^1 )})\label{eq:estimateu_ilocally}
\end{align}
where $C = C(C_i,\chi_i)$. 
Then $u = \sum_{i=1}^4 u_i\circ \tilde{\phi}_i^{-1}$ solves \eqref{eq:linearParaboliconM}.\\
To show the continuity-estimate we first note that the continuity of $u,\partial_t u, \ldots, \partial_x u , \partial_x^2 u, \ldots$ follows directly. Moreover,
\[ \|u\|_{\CalC^0(\Sph^1\times [0,T])} \leq \sum_i \|u_i\|_{\CalC^0(U_i \times [0,T])}\leq C (\|f\|_ {\HS[s]{\frac{s}{4}} ( \Sph^1 \times [0,T]) } + \| u_0 \| _ {\CalC^{4+s}( \Sph^1 )})\]
and the other $\sup$-norms can be estimated similarly. It remains to estimate the H\"older quotients: We distinguish two cases: If $d(x,y)\geq \frac{1}{4}$, then we find from what we have just shown that
\[
 \sup_t \frac{|u(x,t)-u(y,t)|}{d(x,y)^\alpha} \leq \frac
 {2 \|u\|_{\CalC^0}}{4^\alpha} \leq C (\|f\|_ {\HS[s]{\frac{s}{4}} ( \Sph^1 \times [0,T]) } + \| u_0 \| _ {\CalC^{4+s}( \Sph^1 )}),
\]
\[
 \sup_t \frac{|\nabla u(x,t)-\tau_{y,x}\nabla u(y,t)|}{d(x,y)^\alpha} \leq \frac
 {2 \|\nabla u\|_{\CalC^0}}{4^\alpha} \leq C (\|f\|_ {\HS[s]{\frac{s}{4}} ( \Sph^1 \times [0,T]) } + \| u_0 \| _ {\CalC^{4+s}( \Sph^1 )}),
\]
and so on for all higher spatial and time derivatives, since the parallel transport is an isometry. In the second case, if $d(x,y) < \frac{1}{4}$, we find a patch $V_i$ such that $x\in V_i$, $y \in V_i$. Since the chart $\phi_i$ is an isometry we have by \eqref{eq:estimateu_ilocally}
\begin{align*}
 \sup_t \frac{|u(x,t)-u(y,t)|}{d(x,y)^\alpha} &\leq \sup_t \sup_{x_i,y_i\in V_i}  \frac{|u(x_i,t)-u(y_i,t)|}{d(x_i,y_i)^\alpha} \\
 &= \sup_t \sup_{\phi_i(\hat x),\phi_i(\hat y) \in V_i}  \frac{|u(\phi_i(\hat x),t - u(\phi_i(\hat y),t)|}{|\hat x - \hat y|^\alpha}\\
 &= \sup_t [u_i (\cdot,t)]_\alpha \leq C (\|f\|_ {\HS[s]{\frac{s}{4}} ( \Sph^1 \times [0,T]) } + \| u_0 \| _ {\CalC^{4+s}( \Sph^1 )}),
\end{align*}
\begin{align*}
 \sup_t \frac{|\nabla u(x,t)-\tau_{y,x}\nabla u(y,t)|}{d(x,y)^\alpha} 
 &= \sup_t \sup_{\phi_i(\hat x),\phi_i(\hat y) \in V_i}  \frac{|\partial_x u(\phi_i(\hat x),t - \partial_x u(\phi_i(\hat y),t)|}{|\hat x - \hat y|^\alpha}\\
 &= \sup_t [\partial_x u_i (\cdot,t)]_\alpha
 \leq C (\|f\|_ {\HS[s]{\frac{s}{4}} ( \Sph^1 \times [0,T]) } + \| u_0 \| _ {\CalC^{4+s}( \Sph^1 )})
\end{align*}
and so on for all derivatives. This finishes the proof.
\end{proof}

\subsubsection{Uniqueness part of \autoref{thm:nonlinear_schauder}}
\label{proof:uniqueness}
Let us finish the proof of \autoref{thm:nonlinear_schauder} by showing uniqueness of the solution.

\begin{proof}
Let $u_i \in \HS[4+\alpha]{\frac{1+\alpha}{4}}(\Sph^1 \times [0,T_i])$, $i=1,2$ be two solutions to \eqref{eq:schauder_nonlinear}. Without loss of generality we may assume that $0 < T_1 \leq T_2 \leq \varepsilon$. Let us define 
\[
\tau \defeq \sup \{t \in [0,T_1) \,|\, u_1(s) = u_2(s) \text{ in $\CalC^{4,\alpha}(\Sph^1)$ for all } 0 \leq s \leq t\}.\]
We need to show that $\tau = T_1$.\\
We first show that $\tau > 0$. To do so we choose the unique solution $\tilde u \in \HS[4+\alpha]{\frac{1+\alpha}{4}}(\Sph^1 \times [0,T_2])$ to the linear equation \eqref{eq:linear_soln_tilde_u} as in the proof of \autoref{thm:nonlinear_schauder} and let $\eta_i \defeq u_i -\tilde u$. 
From our assumption we find that $\Phi(\eta_i) = 0$, where $\Phi$ is given in \autoref{eq:defPhi}. 
We can estimate the norm of $\|D\Phi[0]\|_{\mathcal{L}(\Xcalnull_T \to \Ycalnull_T)}$ independently of $0 < T \leq \varepsilon$. 
As a consequence of the proof of the Inverse Function Theorem, 
the open sets $U$ and $V$ from \autoref{eq:PhiIso} both contain an open ball centered at zero and $\Phi[0]$ respectively with radius that can be chosen independent of $0 < T \leq \varepsilon$. Hence there exists a constant $r > 0$ independent of $T$ such that $\eta \in U$ if $\|\eta|_{[0,T]}
\|_{\Xcalnull_T} \leq r$ and $g \in V$ 
if $\|g - \Phi[0]|_{[0,T]}\|_{\Ycalnull_T} \leq r$
for all $0 < T \leq \varepsilon$. 
Since 
\begin{equation}
\HS[\alpha]{\frac{\alpha}{4}}(\Sph^1 \times [0,T]) \subset \CalC([0,T];\CalC^{0,\beta}(\Sph^1)) \cap \CalC^{0,\frac{\beta}{4}}([0,T];\CalC(\Sph^1) )   \label{eq:EinbettungStetig}   
\end{equation}
for $\beta < \alpha$ we can show that there exists some $T > 0$ small enough such that 
\begin{equation}\label{eq:smallness_T}
 \|\eta_1|_{[0,T]}
\| _{\Xcalnull_T} \leq r,\quad \|\eta_2|_{[0,T]}
\| _{\Xcalnull_T} \leq r \quad\text{ and }\quad \|\Phi[0]|_{[0,T]}\|_{\Ycalnull_T} \leq r.
\end{equation}
Indeed, since $\Phi[0](t,\cdot)|_{t=0} = 0$ and $\beta < \alpha$ we see that $\|\Phi[0]|_{[0,T]}\|_{\Ycalnull_T} \to 0$ as $ T \to 0$ is a direct consequence of \eqref{eq:EinbettungStetig}. 
To show that also $\|\eta_i|_{[0,T]}
\| _{\Xcalnull_T} \to 0$, we first note that $\eta_i = u_i - \tilde u \in \HS[\alpha + 4]{\frac{1+\alpha}{4}}(\Sph^1 \times [0,T])$ satisfy
$  \eta_i(0) = u_0 - u_0 = 0$  and
\[ \dot \eta_i = \mathbf{F}[u_i] - \mathbf{F}[u_0] + D\mathbf{F}[0]u_0 - D\mathbf{F}[0]\tilde u, \quad\text{
thus }
 \dot \eta_i (0) = 0.
\]
Whence $\|\dot \eta_i\|_\infty \to 0$ and $\|\partial_x^k\eta_i\|_\infty \to 0$ as $T \to 0$, $k = 0, 1, \ldots, 4$, 
hence we can estimate all H\"older seminorms using \eqref{eq:EinbettungStetig} again, showing
\[
 \|\eta_i\|_{\HS[4 + \beta]{\frac{1+\beta}{4}}{(\Sph^1 \times [0,T])}}\to 0   \text{ as }T\to 0.
\]

%
This shows that $\eta_1$ and $\eta_2$ both solve $\Phi(\eta_i) = 0$, and $\eta_1, \eta_2 \in U$ as well as $0 \in V$, whence we find from \eqref{eq:PhiIso} that 
 $\eta_1 = \eta_2 \in \HS[4+\beta]{\frac{1+\beta}{4}}(\Sph^1 \times [0,T])$. From our assumption we know that these functions actually lie in the space $\HS[4+\alpha]{\frac{1+\alpha}{4}}$, thus they coincide in this space, showing $\tau \geq T > 0$.\\
To show that $\tau = T_1$ we assume that $\tau < T_1$ for a moment. Then $u_1(\tau) = u_2(\tau) \eqdef v_0 \in \CalC^{4,\alpha}(\Sph^1)$ by the definition of $\tau$. 
We can again choose a unique $\tilde u \in \HS[4+\alpha]{\frac{1+\alpha}{4}}(\Sph^1 \times [\tau,T_1 - \tau])$ such that $u \defeq \tilde u(\cdot, \cdot - \tau)$ solves the linear equation
\eqref{eq:linear_soln_tilde_u} with initial value $v_0 \in \CalC^{4,\alpha}(\Sph^1 )$. As before, after choosing some $T>0$ small enough we apply \eqref{eq:PhiIso} to find $\eta_1 = \eta_2$ on $[0,T]$, where $\eta_i(x,t) = u_i(x,t+\tau) - \tilde u(x,t) \in \HS[4+ \beta]{\frac{4+\beta}{4}}(\Sph^1 \times [0,T_1-\tau])$.  
Whence we find that $u_1 (t) = u_2(t)$ for $\tau \leq t \leq \tau + T$, contradicting the assumption. This shows that $\tau = T_1$, which finishes the proof.
\end{proof}

\subsubsection{Proof of \autoref{thm:parabolic_smoothing} }
We finish this section with a proof of the parabolic smoothing.
\label{proof:ProofOfSmoothing}
\begin{proof}[{Proof of \autoref{thm:parabolic_smoothing}}]
Let $u\in \HS[4+\alpha]{\frac{4+\alpha}{4}}(\Sph^1\times [0,\varepsilon])$ be the solution of \eqref{eq:schauder_nonlinear} and $0 < \delta < \varepsilon$.\\
 \emph{First step: }Let $\eta_1\colon [0,\varepsilon] \to \R$ be a smooth cut-off function satisfying $\eta_1(t) = 0$ for all $0 \leq t \leq \frac{\delta}{4}$ and $\eta_1(t) = 1$ for all $t \geq \frac{\delta}{2}$. Then the function $w_1 \defeq \eta_1 u \in 
 \HS[4+\alpha]{\frac{4+\alpha}{4}}$ satisfies $w_1(\cdot,0) = 0 \in \CalC^{\infty}(\Sph^1)$ and
 \begin{align*}
  \dot{w_1} &= \dot u \eta_1 + \dot{\eta_1}u = \mathbf{F}[u] \eta_1 + \dot{\eta_1}u\\
  &= - \frac{f_2^4}{|\partial_xf|^4_e}(\partial_x^4 u)\eta_1 +P(\cdot,u,\partial_x u, \partial_x^2 u, \partial_x^3 u,|\partial_xf|_e^{-1})\eta_1 + \dot{\eta_1}u\\
  &=  - \frac{f_2^4}{|\partial_xf|^4_e}\partial_x^4 (u\eta_1) +P(\cdot,u,\partial_x u, \partial_x^2 u, \partial_x^3 u,|\partial_xf|_e^{-1})\eta_1 + \dot{\eta_1}u \eqdef \tilde \alpha(\cdot,u,\partial_xu) \partial_x^4 w_1 + \tilde f_1.
 \end{align*}
 Thus $w_1$ satisfies a linear, parabolic PDE whose coefficients satisfy $\tilde \alpha(\cdot,u,\partial_xu) \in \HS[3+\alpha]{\frac{3+\alpha}{4}} \subset \HS[1+\alpha]{\frac{1+\alpha}{4}}$,  $ P(\cdot,u,\partial_x u, \partial_x^2 u, \partial_x^3 u,|\partial_xf|_e^{-1})\eta_1 \in  \HS[1+\alpha]{\frac{1+\alpha}{4}}$, $\dot{\eta_1}u\in \HS[4+\alpha]{\frac{4+\alpha}{4}} \subset \HS[1+\alpha]{\frac{1+\alpha}{4}}$ and whence
$\tilde f_1  = P\eta_1 + \dot \eta_1 u \in \HS[1+\alpha]{\frac{1+\alpha}{4}}$. 
Thus, by \autoref{thm:linearIsoSchauderM} with $s = 1 + \alpha$, we find that $w = w_1$ is the unique solution $w \in \HS[4+1+\alpha]{\frac{4+1+\alpha}{4}}(\Sph^1 \times [0,\varepsilon])$ to the equation
\begin{equation*}
  \left \{
\begin{array}{rl}
  \partial_t w - \tilde \alpha \partial_x^4 w &= \tilde f_1,\\
 w(0,\cdot) &= 0.
\end{array}
\right.
\end{equation*}
Thus, by the definition of $w_1$ we find 
$u \in
 \HS[5+\alpha]{\frac{5+\alpha}{4}}(\Sph^1\times [\tfrac{1}{2}\delta, \varepsilon] ).$\\
\emph{Second step: }We need to modify 
this argument 
since $u$ (and thus the new function $\tilde f_2$) does not have enough regularity up to $t=0$. Thus we need to shift the problem and consider an initial value at $\frac{1}{2}\delta$.\\
Let $\eta_2\colon [0,\varepsilon] \to \R$ be a smooth cut-off function satisfying $\eta_2(t) = 0$ for all $0 \leq t \leq \frac{2}{3}\delta$ and $\eta_2(t) = 1$ for all $t \geq \frac{3}{4}\delta$. Let $w_2 \defeq \eta_2 u$. Then, by the previous step, $w_2 \in 
 \HS[5+\alpha]{\frac{5+\alpha}{4}}(\Sph^1\times [\tfrac{1}{2}\delta, \varepsilon])$ satisfies $ w_2(\cdot,\tfrac{\delta}{2}) = 0$ and
 \begin{align*}
 \dot{w_2} &= \dot u \eta_2+ \dot{\eta_2}u = \mathbf{F}[u] \eta_1 + \dot{\eta_2}u\\
  &= - \frac{f_2^4}{|\partial_xf|^4_e}(\partial_x^4 u)\eta_2 +P(\cdot,u,\partial_x u, \partial_x^2 u, \partial_x^3 u,|\partial_xf|_e^{-1})\eta_2 + \dot{\eta_2}u\\
  &=  - \frac{f_2^4}{|\partial_xf|^4_e}\partial_x^4 (u\eta_2) +P(\cdot,u,\partial_x u, \partial_x^2 u, \partial_x^3 u,|\partial_xf|_e^{-1})\eta_2 + \dot{\eta_2}u \eqdef \tilde \alpha(\cdot,u,\partial_xu) \partial_x^4 w_2 + \tilde f_2.
 \end{align*}
 Thus $w_2$ satisfies a linear, parabolic PDE whose coefficients satisfy
 $\tilde \alpha(\cdot,u,\partial_xu) \in \HS[4+\alpha]{\frac{4+\alpha}{4}}( \Sph^1\times[\tfrac{1}{2}\delta, \varepsilon]) \subset \HS[2+\alpha]{\frac{2+\alpha}{4}}( \Sph^1\times[\tfrac{1}{2}\delta, \varepsilon])$, $P(\cdot,u,\partial_x u, \partial_x^2 u, \partial_x^3 u,|\partial_xf|_e^{-1})\eta_2 \in \HS[2+\alpha]{\frac{2+\alpha}{4}}( \Sph^1\times[\tfrac{1}{2}\delta, \varepsilon]) ,$ $\dot{\eta_2}u\in \HS[5+\alpha]{\frac{5+\alpha}{4}}( \Sph^1\times[\tfrac{1}{2}\delta, \varepsilon])  \subset \HS[2+\alpha]{\frac{2+\alpha}{4}}( \Sph^1\times[\tfrac{1}{2}\delta, \varepsilon]) $ and whence $\tilde f_2 \in \HS[2+\alpha]{\frac{2+\alpha}{4}}( \Sph^1\times[\tfrac{1}{2}\delta, \varepsilon])$. 
Thus, by \autoref{thm:linearIsoSchauderM} with $s=2+\alpha$, we find  that $w = w_2$ is the unique solution $w \in \HS[4+2+\alpha]{\frac{4+2+\alpha}{4}}(( \Sph^1\times[\tfrac{1}{2}\delta, \varepsilon]) )$ to the equation
\begin{equation*}
  \left \{
\begin{array}{rl}
  \partial_t w - \tilde \alpha \partial_x^4 w &= \tilde f_2,\\
 w(\tfrac12\delta,\cdot) &= 0. 
\end{array}
\right.
\end{equation*}
Thus, by the definition of $w_2$ we find 
$u \in
 \HS[6+\alpha]{\frac{6+\alpha}{4}}( \Sph^1\times [\tfrac{3}{4}\delta, \varepsilon]).$\\
\emph{Third step: }
We successively get
 $u \in
 \HS[n+4+\alpha]{\frac{n+4+\alpha}{4}}(\Sph^1\times[\tfrac{2n-1}{2n}\delta, \varepsilon] ) 
$ 
for all $n \in \N$, whence
\[
 u \in
\bigcap_{n \in \N} \HS[n+4+\alpha]{\frac{n+4+\alpha}{4}}(\Sph^1\times[\tfrac{2n-1}{2n}\delta, \varepsilon]) \subset \CalC^\infty(\Sph^1\times[\delta,\varepsilon]).\qedhere
\]
\end{proof}

\section{Details of interpolation inequalities}
\label{sec:Inter}

Instead of using directly the interpolation inequalities as given in \cite[3.70]{Aubin} we choose here to give the main steps of the derivation in order to keep track of the constants. A detailed proof of the interpolation inequalities in $\R^n$ with respect to $\diff s$ has been given in \cite[App.C]{AnnaPaola1} and we refer partially to those computations. We choose here to give the results for general normal vector fields.

\subsection{\texorpdfstring{$L^p$ as interpolation between $W^{1,2}$ and $L^2$}{Lp as interpolation between Sobolev spaces and L2}.}

In the next lemma we give the main steps of the proof that $L^p$ is the result of an interpolation between $W^{1,2}$ and $L^2$ in a one-dimensional interval. This is \cite[Thm.5.8]{Adams}. We repeat here the main ideas to see how the constant depends on the length.
 
\begin{lemma}\label{lem:interlem1}
Let $L>0$, $p \in [2,\infty]$ and $a=1/2-1/p$ (and $a=1/2$ if $p=\infty$). Then there exists a constant $c$ depending only on $p$ and $1/L$ such that for any smooth function $h:[0,L] \to \R$
$$ \| h \|_{L^p(0,L)} \leq c \| h \|^a_{W^{1,2}(0,L)}\| h \|^{1-a}_{L^2(0,L)}  \, . $$
\end{lemma}
\begin{proof}
By \cite[Lem.4.15]{Adams} for all $x \in [0,L]$ and $r < L/2$ 
$$ |h(x)| \leq \frac{1}{r} \int_{C_{x,r}} |h(y)| \diff y + \int_{C_{x,r}} |h'(y)| \diff y \, ,$$
with $C_{x,r}=[x,x+r]$  if $x \leq L/2$ and $C_{x,r}=[x-r,x]$ otherwise. Denoting by $\chi_{I}$ the characteristic function of a generic interval $I$ we can further estimate $|h(x)|$ as follows
 $$ |h(x)| \leq \frac{1}{r} \chi_{(-r,r)} * (|h|\chi_{[0,L]})(y) + \chi_{(-r,r)} * (|h'|\chi_{[0,L]})(y) ,$$
and by Young's inequality (with $1+1/p=1/2+ (1/2+1/p)$) we find
\begin{align} \nonumber
 \|h\|_{L^p(0,L)} & \leq c(p) (2r)^{\frac12+\frac1p} (  \frac{1}{r} \|h\|_{L^2(0,L)} + \|h'\|_{L^2(0,L)})  \\ \label{eq:ada1}
& \leq 2 c(p) ( r^{\frac12+\frac1p-1} \|h\|_{L^2(0,L)} +  r^{\frac12+\frac1p} \|h\|_{W^{1,2}(0,L)}) \, .  
\end{align}
For $r=\bar{r}:=\|h\|_{L^2(0,L)}/ \|h\|_{W^{1,2}(0,L)}$ the two terms on the right hand side are equal. Hence if $\bar{r} \leq L/3$ we choose $r=\bar{r}$ and the claim follows. Otherwise we take $r=L/3$ in \eqref{eq:ada1} and using that 
$$ \frac{L}{3} \leq \frac{\|h\|_{L^2(0,L)}}{\|h\|_{W^{1,2}(0,L)}} \leq 1, $$
we find
\begin{align*} 
 \|h\|_{L^p(0,L)} 
& \leq 2 c(p) ( (\frac{L}{3})^{\frac12+\frac1p-1} \|h\|^{1-a}_{L^2(0,L)} \|h\|^{a}_{W^{1,2}(0,L)} +  \Big(\frac{\|h\|_{L^2(0,L)}}{\|h\|_{W^{1,2}(0,L)}}\Big)^{\frac12+\frac1p} \|h\|_{W^{1,2}(0,L)}) \\
& \leq c(p,\frac{1}{L})\| h \|^a_{W^{1,2}(0,L)}\| h \|^{1-a}_{L^2(0,L)}  \, , 
\end{align*}
since $1/2+1/p\leq 1$.
\end{proof}

Now the previous result for functions in $\mathbb{S}^1$.

\begin{lemma}\label{lem:interlem2}
Under the assumptions of Proposition \ref{prop:inter1} there exists a constant $c$ depending only on $p$ and $1/L$ such that for any smooth function $h:(\mathbb{S}^1,\diff s) \to \R$
$$ \| h \|_{L^p(\mathbb{S}^1)} \leq c \| h \|^a_{W^{1,2}(\mathbb{S}^1)}\| h \|^{1-a}_{L^2(\mathbb{S}^1)}  \, , $$
with $a=1/2-1/p$ (and $a=1/2$ if $p=\infty$).
\end{lemma}
\begin{proof}
Let $\phi_i:(0,L/2) \to (\mathbb{S}^1,\diff s)$ for $i=1,\ldots,4$ be the isometric charts as defined in \eqref{eq:charts}. Consider $\chi_i$, $i=1,\ldots,4$, be an associated partition of unity on $\mathbb{S}^1$ such that $\| \partial_s \chi\|_{\infty} \leq c_1/L$, $i=1,\ldots,4$. Then $(\chi_i h)\circ \phi_i : (0,L/2) \to \R$ and  we have for $p \in [2,\infty)$ and $i =1,\ldots,4$
\begin{align*}
\| (\chi_i h)\circ \phi_i \|_{L^p(0,L/2)}^p & \leq \int_0^{\frac{L}{2}} |h \circ \phi_i(x)|^p \diff x 
=   \int_{\mathbb{S}^1 \cap V_i} |h (s)|^p \diff s \leq \| h \|_{L^p(\mathbb{S}^1)}^p
\end{align*}
since $\phi_i$ is an isometry. For $p=\infty$ we clearly have $\| (\chi_i h)\circ \phi_i \|_{L^{\infty}(0,L/2)}  \leq \| h \|_{L^{\infty}(\mathbb{S}^1)}$. Similarly, 
\allowdisplaybreaks{\begin{align*}
\| \partial_x((\chi_i h)\circ \phi_i) \|_{L^2(0,L/2)}^2 & \leq 2 \int_0^{\frac{L}{2}} ( |h \circ \phi_i|^2 |\partial_x (\chi_i \circ \phi_i)|^2 + |(\chi_i \circ \phi_i)|^2| \partial_x (h \circ \phi_i)|^2) \diff x \\
& \leq 2 (\frac{c_1^2}{L^2} +1)\int_0^{\frac{L}{2}} ( |h \circ \phi_i(x)|^2 + | \partial_x (h \circ \phi_i)(x)|^2) \diff x \\
& \leq 2 (\frac{c_1^2}{L^2} +1) \int_{\mathbb{S}^1 \cap V_i}( |h (s)|^2 + | \partial_s h |^2) \diff s 
\\&
\leq 2 (\frac{c_1^2}{L^2} +1) \| h\|_{W^{1,2}(\mathbb{S}^1)}^2 \, .
\end{align*}}
Hence we conclude for $p \in [2,\infty)$ using Lemma \ref{lem:interlem1} and the previous estimates that
\begin{align*}
\| h\|_{L^p(\mathbb{S}^1)} & \leq \sum_{i=1}^4 \| \chi_i h\|_{L^p(\mathbb{S}^1)} = \sum_{i=1}^4 (\int_0^{\frac{L}{2}} |((\chi_i h)\circ \phi_i)|^p  \diff x )^{\frac{1}{p}} \\
& \leq \sum_{i=1}^4 c_2(p,\frac{1}{L}) \| (\chi_i h)\circ \phi_i\|^a_{W^{1,2}(0,L)}\| (\chi_i h)\circ \phi_i \|^{1-a}_{L^2(0,L)} \\
& \leq c_3(p,\frac{1}{L}) \| h\|^a_{W^{1,2}(\mathbb{S}^1)}\| h \|^{1-a}_{L^2(\mathbb{S}^1)} \, ,
\end{align*}
with $a=1/2-1/p$. For $p=\infty$ we first observe that for all $x \in \mathbb{S}^1$ there exists an $i \in \{1, ..,4\}$ such that $\chi_i(x)\geq 1/4$, hence
\begin{align*}
\| h\|_{L^{\infty}(\mathbb{S}^1)} & \leq 4 \max_{i=1,..,4} \| \chi_i h\|_{L^{\infty}(\mathbb{S}^1)} = 4 \max_{i=1,..,4} \| (\chi_i h) \circ \phi_i\|_{L^{\infty}(0,L/2)}  \\
& \leq 4\max_{i=1,..,4} c_4 \| (\chi_i h)\circ \phi_i\|^a_{W^{1,2}(0,L)}\| (\chi_i h)\circ \phi_i \|^{1-a}_{L^2(0,L)} \, ,
\end{align*}
and the claim follows as above.
\end{proof}

Here we give the precise statement that $L^p$ is an interpolation between $W^{1,2}$ and $L^2$ for normal  vector fields. 
\begin{prop}\label{prop:inter1}
Let $f:\mathbb{S}^1 \to \Hyp^2 $ be a smooth immersion such that $\int_{\mathbb{S}^1} \diff s=L>0$ with $\diff s = |\partial_x f|_g \diff x$. Then for any $p \in [2,\infty]$ there exists a constant $C$ depending only on $p$ and $\frac{1}{L}$ such that for any smooth {normal} vector field $\Phi:(\mathbb{S}^1,\diff s)  \to TM$ we have
$$ \| \Phi \|_{L^p(\mathbb{S}^1)} \leq c \| \Phi \|^a_{W^{1,2}(\mathbb{S}^1)}\| \Phi \|^{1-a}_{L^2(\mathbb{S}^1)}  \, , $$
with $a=1/2-1/p$ (and $a=1/2$ if $p=\infty$).
\end{prop}
\begin{proof} If $\Phi \ne 0$ on $\mathbb{S}^1$ then $| \Phi|_g$ is a smooth function on $\mathbb{S}^1$ and by Lemma \ref{lem:interlem2} we find
$$ \| \Phi \|_{L^p(\mathbb{S}^1)} = \| | \Phi|_g \|_{L^p(\mathbb{S}^1)} \leq c \| | \Phi|_g \|^a_{W^{1,2}(\mathbb{S}^1)}\| | \Phi|_g \|^{1-a}_{L^2(\mathbb{S}^1)}  \, , $$
and the claim follows in this case since $\| | \Phi|_g \|_{L^2(\mathbb{S}^1)}  = \|  \Phi \|_{L^2(\mathbb{S}^1)}$ and being $\Phi$ normal
\begin{align*}
|\partial_s| \Phi|_g | & = \frac{|\partial_x | \Phi|_g|}{|\partial_x f|_g} 
= \frac{|\langle \Phi, \nabla_{\partial_x} \Phi\rangle_g|}{| \Phi|_g|\partial_x f|_g} 
= \frac{|\langle \Phi, \nabla_{\partial_x}^\bot \Phi\rangle_g|}{| \Phi|_g|\partial_x f|_g}\leq
\frac{| \nabla_{\partial_x}^\bot \Phi|_g}{|\partial_x f|_g} = | \nabla_{\partial_s}^\bot \Phi|_g
\end{align*} 
and hence $\| | \Phi|_g \|_{W^{1,2}(\mathbb{S}^1)}  \leq \|  \Phi \|_{W^{1,2}(\mathbb{S}^1)}$. 
If $\Phi=0$ somewhere then we get back to the previous case with an approximation argument.
\end{proof}

\subsection{The general interpolation inequality}

\begin{lemma}\label{lem:lemC5}
Consider the same assumptions of Proposition \ref{prop:inter1}. Let $\Phi: (\mathbb{S}^1, \diff s) \to T \Hyp^2$ be a smooth {normal} vector field. Then for any $k\geq 2$, $k \in \N$ and $0<i<k$ there exists a constant $c$ depending only on $i$ and $k$ such that for any $\varepsilon \in (0,1)$
$$ \| (\nabla_{\partial_s}^\bot)^i \Phi \|_{L^2(\mathbb{S}^1)} \leq c ( \varepsilon \| \Phi \|_{W^{k,2}(\mathbb{S}^1)} + \varepsilon^{\frac{i}{i-k}}\| \Phi \|_{L^2(\mathbb{S}^1)})  \, , $$
and for $0 \leq i <k$
$$ \| (\nabla_{\partial_s}^\bot)^i \Phi \|_{L^2(\mathbb{S}^1)} \leq c \| \Phi \|^{\frac{i}{k}}_{W^{k,2}(\mathbb{S}^1)} \| \Phi \|^{\frac{k-i}{k}}_{L^2(\mathbb{S}^1)}  \, . $$
\end{lemma}
\begin{proof}
The second inequality follows from the first choosing $\varepsilon$ such that the two terms on the right hand side are equal. Notice that the second inequality is trivially satisfied for $i=0$ taking simply $c\geq 1$.

It remains to prove the first inequality. Since $\mathbb{S}^1$ has no boundary and $\Phi$ is a normal vector field, using \eqref{eq:PartialIntegration} we find for $k=2$ and $i=1$
\begin{align*}
\| \nabla_{\partial_s}^\bot \Phi \|^2_{L^2(\mathbb{S}^1)} & = \int_{\mathbb{S}^1} \langle \nabla_{\partial_s} \Phi, \nabla_{\partial_s}^\bot \Phi \rangle_g \diff s = \int_{\mathbb{S}^1} ( \partial_s \langle  \Phi, \nabla_{\partial_s}^\bot \Phi \rangle_g - \langle  \Phi, (\nabla^\bot_{\partial_s})^2 \Phi \rangle_g) \diff s  \\
& \leq \| \Phi \|_{L^2(\mathbb{S}^1)} \| (\nabla^\bot_{\partial_s})^2 \Phi\|_{L^{2}(\mathbb{S}^1)}  \leq \| \Phi \|_{L^2(\mathbb{S}^1)} \| \Phi \|_{W^{2,2}(\mathbb{S}^1)} \\
& \leq \frac12 ( \varepsilon^2  \| \Phi \|^2_{W^{2,2}(\mathbb{S}^1)} + \varepsilon^{-2}  \| \Phi \|^2_{L^2(\mathbb{S}^1)} )  \leq \frac12 ( \varepsilon  \| \Phi \|_{W^{2,2}(\mathbb{S}^1)} + \varepsilon^{-1}  \| \Phi \|_{L^2(\mathbb{S}^1)} )^2 \, .
\end{align*} 
The rest of the proof is by induction and the details are as in \cite[Lem. C.5]{AnnaPaola1}.
\end{proof}

\begin{lemma}\label{lem:lemC6}
Assume the assumptions of Proposition \ref{prop:inter1}. Then for any $k \in \N$, $0\leq i<k$ and $p\in[2, \infty]$ there exists a constant $c$ depending only on $i$, $k$, $p$ and $1/L$  such that
$$ \| (\nabla_{\partial_s}^\bot)^i \Phi \|_{L^p(\mathbb{S}^1)} \leq c \| (\nabla_{\partial_s}^\bot)^i \Phi  \|_{W^{k-i,2}(\mathbb{S}^1)}^{\frac{1}{k-i}(\frac12 -\frac1p)}  \|(\nabla_{\partial_s}^\bot)^i \Phi \|_{L^2(\mathbb{S}^1)}^{1-\frac{1}{k-i}(\frac12 -\frac1p)} \,  . $$
Here $1/p:=0$ if $p=\infty$.
\end{lemma}
\begin{proof}
If $k-i=1$ this is Proposition \ref{prop:inter1}. Otherwise the estimate can be proved applying first Proposition \ref{prop:inter1} and then Lemma \ref{lem:lemC5}. The details are as in \cite[Lem. C.6]{AnnaPaola1}.
\end{proof}

We can finally prove Proposition \ref{prop:propsonntag}.

\begin{proof}[Proof of Proposition \ref{prop:propsonntag}]
If $i=0$ and $k=1$ this is Lemma \ref{lem:lemC6} (or Proposition \ref{prop:inter1}) with $\Phi=\Kapp$. If $k\geq 2$ and $0 \leq i <k$ with Lemma \ref{lem:lemC6} 
\begin{align*}
\| (\nabla_{\partial_s}^\bot)^i \Kapp \|_{L^p(\mathbb{S}^1)} & \leq c \| (\nabla_{\partial_s}^\bot)^i \Kapp  \|_{W^{k-i,2}(\mathbb{S}^1)}^{\frac{1}{k-i}(\frac12 -\frac1p)}  \|(\nabla_{\partial_s}^\bot)^i \Kapp \|_{L^2(\mathbb{S}^1)}^{1-\frac{1}{k-i}(\frac12 -\frac1p)} \\
& \leq c \|  \Kapp  \|_{W^{k,2}(\mathbb{S}^1)}^{\frac{1}{k-i}(\frac12 -\frac1p)}  \|(\nabla_{\partial_s}^\bot)^i \Kapp \|_{L^2(\mathbb{S}^1)}^{1-\frac{1}{k-i}(\frac12 -\frac1p)}
\end{align*}
with $c=c(i,k,p,1/L)$ and $1/p:=0$ if $p=\infty$. Then from the second statement in Lemma \ref{lem:lemC5} we find
$$  \|(\nabla_{\partial_s}^\bot)^i \Kapp \|_{L^2(\mathbb{S}^1)} \leq c \| \Kapp \|^{\frac{i}{k}}_{W^{k,2}(\mathbb{S}^1)} \| \Kapp \|^{\frac{k-i}{k}}_{L^2(\mathbb{S}^1)} ,$$ 
with $c=c(i,k,p,1/L)$. Combining these two inequalities the claim follows.
\end{proof}


\begin{thebibliography}{DFGS11}

\bibitem[AF03]{Adams}
Robert~A. Adams and John J.~F. Fournier.
\newblock {\em Sobolev spaces}, volume 140 of {\em Pure and Applied Mathematics
  (Amsterdam)}.
\newblock Elsevier/Academic Press, Amsterdam, second edition, 2003.

\bibitem[Aub82]{Aubin}
Thierry Aubin.
\newblock {\em Nonlinear analysis on manifolds. {M}onge-{A}mp\`ere equations},
  volume 252 of {\em Grundlehren der Mathematischen Wissenschaften [Fundamental
  Principles of Mathematical Sciences]}.
\newblock Springer-Verlag, New York, 1982.

\bibitem[BDF10]{MR2729304}
Matthias Bergner, Anna Dall'Acqua, and Steffen Fr\"ohlich.
\newblock Symmetric {W}illmore surfaces of revolution satisfying natural
  boundary conditions.
\newblock {\em Calc. Var. Partial Differential Equations}, 39(3-4):361--378,
  2010.

\bibitem[BG86]{BrGr}
Robert Bryant and Phillip Griffiths.
\newblock Reduction for constrained variational problems and {$\int k^2/2
  \,ds$}.
\newblock {\em Amer. J. Math.}, 108(3):525--570, 1986.

\bibitem[Bla09]{Blatt}
Simon Blatt.
\newblock A singular example for the {W}illmore flow.
\newblock {\em Analysis (Munich)}, 29(4):407--430, 2009.

\bibitem[dC76]{doCarmo1}
Manfredo~P. do~Carmo.
\newblock {\em Differential geometry of curves and surfaces}.
\newblock Prentice-Hall, Inc., Englewood Cliffs, N.J., 1976.
\newblock Translated from the Portuguese.

\bibitem[dC92]{doCarmo}
Manfredo~Perdig{\~a}o do~Carmo.
\newblock {\em Riemannian geometry}.
\newblock Mathematics: Theory \& Applications. Birkh\"auser Boston, Inc.,
  Boston, MA, 1992.
\newblock Translated from the second Portuguese edition by Francis Flaherty.

\bibitem[DDG08]{MR2480063}
Anna Dall'Acqua, Klaus Deckelnick, and Hans-Christoph Grunau.
\newblock Classical solutions to the {D}irichlet problem for {W}illmore
  surfaces of revolution.
\newblock {\em Adv. Calc. Var.}, 1(4):379--397, 2008.

\bibitem[DFGS11]{MR2770424}
Anna Dall'Acqua, Steffen Fr\"ohlich, Hans-Christoph Grunau, and Friedhelm
  Schieweck.
\newblock Symmetric {W}illmore surfaces of revolution satisfying arbitrary
  {D}irichlet boundary data.
\newblock {\em Adv. Calc. Var.}, 4(1):1--81, 2011.

\bibitem[DKS02]{DKS}
Gerhard Dziuk, Ernst Kuwert, and Reiner Sch\"atzle.
\newblock Evolution of elastic curves in {$\mathbb{R}^n$}: existence and
  computation.
\newblock {\em SIAM J. Math. Anal.}, 33(5):1228--1245, 2002.

\bibitem[DLP16]{ChunAnnaPaola}
Anna Dall'Acqua, Chun-Chi Lin, and Paola Pozzi.
\newblock A gradient flow for open elastic curves with fixed length and clamped
  ends.
\newblock {\em Ann. Sc. Norm. Super. Pisa Cl. Sci.}, accepted, 2016.

\bibitem[DP14]{AnnaPaola1}
Anna Dall'Acqua and Paola Pozzi.
\newblock A {W}illmore-{H}elfrich {$L^2$}-flow of curves with natural boundary
  conditions.
\newblock {\em Comm. Anal. Geom.}, 22(4):617--669, 2014.

\bibitem[EZ98]{Eidelman}
Samuil~D. Eidelman and Nicolae~V. Zhitarashu.
\newblock {\em Parabolic boundary value problems}, volume 101 of {\em Operator
  Theory: Advances and Applications}.
\newblock Birkh\"auser Verlag, Basel, 1998.
\newblock Translated from the Russian original by Gennady Pasechnik and Andrei
  Iacob.

\bibitem[Ger06]{Gerhardt}
Claus Gerhardt.
\newblock {\em Curvature problems}, volume~39 of {\em Series in Geometry and
  Topology}.
\newblock International Press, Somerville, MA, 2006.

\bibitem[Koi96]{Koiso}
Norihito Koiso.
\newblock On the motion of a curve towards elastica.
\newblock In {\em Actes de la {T}able {R}onde de {G}\'eom\'etrie
  {D}iff\'erentielle ({L}uminy, 1992)}, volume~1 of {\em S\'emin. Congr.},
  pages 403--436. Soc. Math. France, Paris, 1996.

\bibitem[Lan83]{Lang_Analysis}
Serge Lang.
\newblock {\em Real analysis}.
\newblock Addison-Wesley Publishing Company, Advanced Book Program, Reading,
  MA, second edition, 1983.

\bibitem[Lin12]{MR2911840}
Chun-Chi Lin.
\newblock {$L^2$}-flow of elastic curves with clamped boundary conditions.
\newblock {\em J. Differential Equations}, 252(12):6414--6428, 2012.

\bibitem[LS84a]{LangerSinger2}
Joel Langer and David Singer.
\newblock Curves in the hyperbolic plane and mean curvature of tori in
  {$3$}-space.
\newblock {\em Bull. London Math. Soc.}, 16(5):531--534, 1984.

\bibitem[LS84b]{LangerSinger1}
Joel Langer and David~A. Singer.
\newblock The total squared curvature of closed curves.
\newblock {\em J. Differential Geom.}, 20(1):1--22, 1984.

\bibitem[LSU68]{MR0241821}
O.~A. Lady{\v{z}}enskaja, V.~A. Solonnikov, and N.~N. Ural'ceva.
\newblock {\em {L}inear and {Q}uasi-linear {E}quations of {P}arabolic {T}ype}.
\newblock Translations of Mathematical Monographs, Vol. 23 American
  Mathematical Society, 1968.

\bibitem[{Man}17]{2017arXiv170502177M}
R.~{Mandel}.
\newblock {Explicit formulas and symmetry breaking for Willmore surfaces of
  revolution}.
\newblock {\em ArXiv e-prints}, May 2017.

\bibitem[Mum94]{Mumford1994}
David Mumford.
\newblock {\em Elastica and Computer Vision}, pages 491--506.
\newblock Springer New York, New York, NY, 1994.

\bibitem[NO14]{OkabeNovaga}
Matteo Novaga and Shinya Okabe.
\newblock Curve shortening-straightening flow for non-closed planar curves with
  infinite length.
\newblock {\em J. Differential Equations}, 256(3):1093--1132, 2014.

\bibitem[Spe17]{STE_LP}
Adrian Spener.
\newblock Short time existence for the elastic flow of clamped curves.
\newblock {\em Mathematische Nachrichten}, pages n/a--n/a, 2017.

\bibitem[Sze68]{Szenthe}
J.~Szenthe.
\newblock On the total curvature of closed curves in {R}iemannian manifolds.
\newblock {\em Publ. Math. Debrecen}, 15:99--105, 1968.

\bibitem[Tru83]{MR714991}
C.~Truesdell.
\newblock The influence of elasticity on analysis: the classic heritage.
\newblock {\em Bull. Amer. Math. Soc. (N.S.)}, 9(3):293--310, 1983.

\bibitem[Tsu74]{Tsukamoto}
Y\^otar\^o Tsukamoto.
\newblock On the total absolute curvature of closed curves in manifolds of
  negative curvature.
\newblock {\em Math. Ann.}, 210:313--319, 1974.

\bibitem[Wen93]{Wen}
Yingzhong Wen.
\newblock {$L^2$} flow of curve straightening in the plane.
\newblock {\em Duke Math. J.}, 70(3):683--698, 1993.

\bibitem[Zei86]{ZeidlerI}
Eberhard Zeidler.
\newblock {\em Nonlinear functional analysis and its applications. {I}}.
\newblock Springer-Verlag, New York, 1986.
\newblock Fixed-point theorems, Translated from the German by Peter R. Wadsack.

\end{thebibliography}
\bibliographystyle{alpha}

\end{document}